\documentclass{siamltex}
\usepackage{graphicx}
\usepackage{latexsym}
\usepackage{subfigure}
\usepackage{amsmath}
\usepackage{enumerate}
\usepackage{amssymb}
\usepackage[mathscr]{eucal}
\usepackage{color}
\usepackage{array}
\newtheorem{thm}{Theorem}
\newtheorem{cor}[thm]{Corollary}
\newtheorem{lem}[thm]{Lemma}

\newtheorem{defn}[thm]{Definition}
\newtheorem{rem}[thm]{Remark}

\newtheorem{alg}{Algorithm}

\numberwithin{equation}{section}
\newcommand{\R}{\mathcal{R}}
\newcommand{\RR}{\mathbb{R}}
\newcommand{\CC}{\mathbb{C}}
\renewcommand{\L}{\mathcal{L}}
\renewcommand{\P}{\mathcal{P}}
\newcommand{\NN}{\mathbb{N}}
\newcommand{\N}{\mathbb{N}}
\newcommand{\Z}{\mathbb{Z}}

\newcommand{\eps}{\epsilon}

\newcommand{\dps}{\displaystyle}

\renewcommand{\Re}{\mathrm{Re}}

\newcommand{\emat}{\end{pmatrix}}

\DeclareMathOperator{\Id}{{\rm Id}}

\newcommand{\G}{\mathcal{G}}
\newcommand{\F}{\mathcal{F}}
\newcommand{\C}{\mathcal{C}}

\begin{document}

% \figcapson
% \printfigures

\title{A micro-macro parareal algorithm: application to singularly perturbed ordinary differential equations}

\author{Fr\'ed\'eric Legoll\thanks{Laboratoire Navier, \'Ecole Nationale des Ponts et Chauss\'ees,
Universit\'e Paris-Est, 6 et 8 avenue Blaise Pascal, 77455
Marne-La-Vall\'ee Cedex 2, France; INRIA Rocquencourt, MICMAC team-project, Domaine de Voluceau, B.P. 105,
78153 Le Chesnay Cedex, France
}
\and Tony Leli\`evre\thanks{CERMICS, \'Ecole Nationale des Ponts et Chauss\'ees, Universit\'e
Paris-Est, 6 et 8 avenue Blaise Pascal, 77455 Marne-La-Vall\'ee Cedex 2,
France; INRIA Rocquencourt, MICMAC team-project, Domaine de Voluceau, B.P. 105,
78153 Le Chesnay Cedex, France}
\and Giovanni Samaey\thanks{Scientific Computing, Department of Computer Science, KU Leuven,
Celestijnenlaan 200A, 3001 Leuven, Belgium
}
Scientific Computing, Department of Computer Science, KU Leuven,
Celestijnenlaan 200A, 3001 Leuven, Belgium}

\maketitle

\begin{abstract}
We introduce a micro-macro parareal algorithm for the time-parallel
integration of multiscale-in-time systems. The algorithm first computes
a cheap, but inaccurate, solution using a coarse propagator (simulating an
approximate slow macroscopic model), which is iteratively corrected
using a fine-scale propagator (accurately simulating the full
microscopic dynamics). This correction is done in parallel over many
subintervals, thereby reducing the wall-clock time needed to obtain the
solution, compared to the integration of the full microscopic model over
the complete time interval.
We provide a numerical analysis of the algorithm for a prototypical
example of a micro-macro model, namely singularly perturbed ordinary
differential equations. We show that the computed solution are better
and better approximations of the full microscopic solution (when the parareal iterations proceed)
only if special care is taken during the coupling of the microscopic and
macroscopic levels of description. The error bound depends on the
modeling error of the approximate macroscopic model. We illustrate these
results with numerical experiments.
\end{abstract}

\maketitle

\section{Introduction}

In many applications, a system is modeled using a high-dimensional system of
differential equations that captures phenomena occurring at multiple time
scales. Unfortunately, the computational cost of simulating such fine-scale
systems (which we call \emph{microscopic} in this work) on macroscopic
time intervals is prohibitive, and one often resorts to
low-dimensional, coarse-grained, effective models (which we call
\emph{macroscopic}), in which the fast 
degrees of freedom are eliminated. Many methods have been proposed to
obtain such macroscopic models, either analytically (see
e.g.~\cite{Pavliotis2008} for a recent overview) or numerically.
We refer, for instance, to the work on 
equation-free~\cite{KevrGearHymKevrRunTheo03,Kevrekidis:2009p7484} or
heterogeneous multiscale methods~\cite{EEng03,E:2007p3747}, and
references therein. However, by construction, these macroscopic models
only capture the original full microscopic dynamics {\em approximately}. 

Here, we present and analyze a numerical
multiscale method that aims at efficiently simulating the full
microscopic dynamics (and not a macroscopic approximation of it)
over long time intervals, using an effective (approximate) macroscopic
model as a predictor and the microscopic model as a corrector.
To this end, we propose a micro-macro version of the parareal
algorithm~\cite{Lions2001}. The parareal algorithm was originally proposed to
solve time-dependent problems using computations in parallel, aiming at
exploiting the presence 
of multiple processors to reduce the real (wall-clock) time needed to obtain
a solution on a long time interval. It is based on a decomposition of
the time interval into subintervals, and makes use of a
predictor-corrector strategy, in which the calculation of the
corrections is performed concurrently on the different processors that
are available. In what follows, we propose a version of this
algorithm well-adapted to our multiscale-in-time context. 

For the sake of clarity, and to
better describe our aim, we now present the
parareal algorithm in some detail. To fix the ideas, assume that the
problem at hand is
\begin{equation}
\label{eq:problem-1order}
\frac{du}{dt} = f(u), \quad u(0) = u_0, \quad u(t) \in \RR^d,
\quad t \in [0,T],
\end{equation}
the exact flow of which is denoted $u(t) = \mathcal{E}_t(u_0)$. Suppose that 
we have at hand two propagators to integrate~\eqref{eq:problem-1order},
$\F_{\Delta t}$ and $\G_{\Delta t}$. The propagator
$\F_{\Delta t}$ is a fine, expensive propagator, which accurately
approximates the exact flow $\mathcal{E}_{\Delta t}$ over the time range
$\Delta t$,
%  For instance, this
% propagator amounts to using many steps of a one-step integrator to
% advance the solution over the time range $\Delta t$:
% $$
% F_{\Delta t} = 
% \left( \psi_{\delta t} \right)^{(\Delta t/\delta t)},
% $$
% where $\psi_{\delta t}$ is a one-step integrator of step size $\delta t$.
whereas the propagator $\G_{\Delta t}$ is a coarse propagator, 
which is a less accurate approximation of the exact flow. In turn,
$\G_{\Delta t}$ is less expensive to 
simulate than $\F_{\Delta t}$.
%For example, $G_{\Delta t}$ may be built using the same
%discretization scheme as for $F_{\Delta t}$, but with a much larger time
%step.  We will assume the cost of $G_{\Delta t}$ to be negligible
%compared to the computation of $F_{\Delta t}$.
%
% FL: je reformule la phrase ci-dessus
% De plus, la discussion des couts se fait apres
For example, $\F_{\Delta t}$ and $\G_{\Delta t}$ may correspond to
integrating~\eqref{eq:problem-1order} over the time range $\Delta t$
with a given discretization scheme, using either a small time step (for
$\F_{\Delta t}$) or a large time step (for $\G_{\Delta t}$). 
%% CHANGED: made the explanation a lot shorter.  Please check if message is still conveyed.
%  $d T \gg \delta t$. In this case,
% $$
% G_{\Delta t} = 
% \left( \psi_{d T} \right)^{(\Delta t/dT)},
% $$
% and computing $G_{\Delta t}$ amounts to performing $\Delta
% t/dT$ time steps of length $dT$, rather than $\Delta t/\delta t$ time
% steps of length $\delta t$ for $F_{\Delta t}$. As $d T \gg
% \delta t$, computing $G_{\Delta t}$ is cheaper than computing 
% $F_{\Delta t}$.
The parareal algorithm iteratively constructs a sequence of $N$-tuples
${\bf u}_k \equiv \left\{ u^n_k \right\}_{1 \leq n \leq N}$ (with $N =
T/\Delta t$), such that, at every iteration $k\ge 0$, $u^n_k$ is an
approximation of $u(n \Delta t)$. For $k=0$, the initial approximation is
obtained using the coarse propagator $\G_{\Delta t}$:
$$
u^{n+1}_{k=0} = \G_{\Delta t}(u^n_{k=0}), \quad
u^0_{k=0} = u_0.
$$
In the subsequent parareal iterations, the approximation is corrected
using  
\begin{equation}
\label{eq:scheme-parareal}
u^{n+1}_{k+1} = \G_{\Delta t}(u^n_{k+1}) +  
\F_{\Delta t}(u^n_k)
- \G_{\Delta t}(u^n_k),
\end{equation}
with the initial condition $u^0_{k+1} = u_0$. The solution
to~\eqref{eq:scheme-parareal} can be very efficiently computed using
the following procedure.
Once the solution at parareal iteration $k$ has been computed, 
we first compute the corrections 
$\F_{\Delta t}(u^n_k) - \G_{\Delta t}(u^n_k)$
in {\em parallel} over each subinterval $[n \Delta t,(n+1) \Delta t]$, 
$0 \leq n \leq N-1$. We then only need to propagate
these corrections sequentially, by adding  $\G_{\Delta t}(u^n_{k+1})$
to the stored correction 
$\F_{\Delta t}(u^n_k) - \G_{\Delta t}(u^n_k)$. This yields the
solution at parareal iteration $k+1$.
 
It
has been shown (see
e.g.~\cite{baffico2002parallel,bal23parareal,bal2005convergence,Lions2001,maday-08}) that,
when $k$ goes to infinity, the parareal solution converges to the
reference solution, namely the solution given by the fine-scale
propagator $\F_{\Delta t}$ used in a 
sequential fashion from the initial condition: 
\begin{equation}
\label{eq:conv}
\forall n, \ 0 \leq n \leq T/\Delta t, \quad
\lim_{k \to \infty} u^n_k = \F_{\Delta t}^{(n)}(u_0).
\end{equation}
The computational gain of the parareal algorithm stems from the fact
that, in~\eqref{eq:scheme-parareal}, the accurate simulations (using the fine-scale propagator $\F_{\Delta t}$)
are decoupled one from each other, and can therefore be executed in
parallel on different processors. Suppose that the cost of a single
evaluation of $\F_{\Delta t}$ is much larger than the cost of propagating
the system according to $\G_{\Delta t}$ over the complete time range
$[0,T]$. 
Assuming the cost of the
fine-scale propagator $\F_{\Delta t}$ to be proportional to $\Delta t$,
the cost of $K$ iterations of the parareal algorithm is proportional to
$K\Delta t$. This cost is to be compared to the cost of computing the
reference solution using the fine-scale propagator sequentially, which
is proportional to $N\Delta t$. The computational speed-up is thus
$N/K$, which is larger than one if the number $K$ of parareal iterations
to obtain convergence in~\eqref{eq:conv} is small enough. 

In this article, we propose and analyze a micro-macro version of the
parareal algorithm. We assume that the variables in the 
microscopic model can be split into slow and fast components, and that 
we have at hand an approximate macroscopic model for the slow
components under some time scale separation
assumption (see Section~\ref{sec:model} for the precise model we
consider here). In this setting, we will use the parareal algorithm 
where the fine-scale propagator $\F_{\Delta t}$ is an integrator
for the high-dimensional microscopic model, whereas the coarse
propagator, here denoted $\C_{\Delta t}$, is an 
integrator of the {\em low-dimensional}, approximate macroscopic model
(we use the notation $\C_{\Delta t}$ rather than $\G_{\Delta t}$ to
emphasize the fact that our coarse integrator acts on a system of
smaller dimension than the reference one). The
novelty therefore is to simultaneously use two models at different
levels of description, rather than two discretizations of the same
model. The cost of the coarse propagator is typically negligible
for two reasons:
(i) the macroscopic model
only contains the slow components of the evolution, and therefore allows
for a larger time step; and (ii) the macroscopic model is
low-dimensional, and therefore requires less work per time
step.
Again, the
aim of the micro-macro parareal method is to speed up
the computations (compared to a full microscopic simulation) by
allowing the microscopic simulations starting from the different
intermediate time instances $n \Delta t$ to be performed in parallel
over each subinterval $[n \Delta t,(n+1) \Delta t]$, with $0 \leq n \leq N-1$. 

As a model problem, we take the setting of singularly perturbed systems
of ODEs. Such a model problem is a widely accepted first test case when
proposing algorithms for problems with time-scale separation, see
e.g.~\cite{Givon2004}. We perform a numerical analysis of the algorithm
we propose in a linear setting (see Section~\ref{sec:model} for the
description of the model problem, and Section~\ref{sec:analysis} for the
numerical analysis), and illustrate these results by numerical
simulations in Section~\ref{sec:num}. However, our algorithm is {\em not}
restricted to the linear setting, and we numerically observe in
Section~\ref{sec:nonlin} that it performs equally well on a
nonlinear test-case. 

\medskip

Since its introduction in~\cite{Lions2001}, the parareal strategy has
been applied to a wide range of problems, including fluid-structure
interaction~\cite{farhat2003time}, Navier--Stokes equation
simulation~\cite{fischer2005parareal}, reservoir
simulation~\cite{garrido2005}, etc.
The algorithm has been further analyzed
in~\cite{maday2002parareal,maday2005parareal}. Its stability has been
investigated in~\cite{bal2005convergence,staff2005stability}. An
alternative formulation of the algorithm has been proposed
in~\cite{bal23parareal}, or, equivalently, in~\cite{baffico2002parallel} 
in a simplified setting. We refer to~\cite{gander2007analysis} for a
reformulation in a more general setting that relates the parareal
strategy to earlier time-parallel algorithms, such as multiple shooting
(see e.g.~\cite{keller1968numerical,nievergelt1964parallel}) or
multigrid waveform relaxation (see
e.g.~\cite{lubich1987multi,vandewalle1992efficient}) approaches. Several
variants of the algorithm have been proposed, for instance
in~\cite{dai2010symmetric,farhat2003time,garrido2006convergent} (see
also~\cite{bal2003parallelization} in the context of
stochastic differential equations). The numerical
analysis of the algorithm has been first performed for linear
initial-value problems. A numerical analysis in a nonlinear context has
been proposed in~\cite{gander2008nonlinear}.

A micro-macro version of the parareal algorithm, similar to what is
presented in this article, 
has already been considered in a number of works. The authors
of~\cite{BBK,maday41parareal} consider a singularly perturbed system of
ordinary differential equations (ODEs) at the microscopic level and the
limiting differential-algebraic equation (DAE) at the macroscopic
level. In these two works, the coarse propagator contains {\em all} degrees of
freedom in the system. The slow degrees of freedom are evolved according
to a differential equation, and the fast degrees of freedom are evolved
using algebraic constraints (they somehow instantaneously adapt to the
values of the slow degrees of freedom). In contrast, our approach
completely eliminates the fast variables from the coarse propagator, and
only evolves the slow variables. This difference has a number of
consequences:
\begin{itemize}
\item The coarse propagator in the algorithms proposed here is cheaper
than that of~\cite{BBK,maday41parareal} (because it contains less
degrees of freedom) and more convenient (because the coarse propagator
simulates an ODE rather than a DAE); 
\item The algorithms proposed here require operators to reconstruct
  microscopic states from macroscopic ones, while the algorithm
  in~\cite{BBK,maday41parareal} can simply use the parareal
  iteration~\eqref{eq:scheme-parareal}. This also influences the
  convergence behavior.
\end{itemize}
A detailed comparison between our algorithms and that proposed
in~\cite{BBK,maday41parareal} is given in
Section~\ref{sec:compare_scheme}.
%% CHANGED: added last part of sentence.
%% NOTE: differences are :
% 1) manner of the coupling of the two models 
% 2) they have no analysis
% 3) their numerical results have discretization at the macrolevel
% 4) their results show convergence to something, but it is not clear if this is the exact microdynamics
% 5) their results might correspond to our results for the bad algorithm, but there is no way to tell.

Other micro-macro parareal algorithms have also been proposed, in
contexts different from ours. In~\cite{engblom2009parallel}, a parareal
algorithm for multiscale stochastic chemical kinetics is presented, in
which the macroscopic level uses the mean-field limiting
ODE. In~\cite{mitran2010time}, the parareal algorithm is used with
a kinetic Monte Carlo model at the macroscopic level and molecular
dynamics at the microscopic level. 

\medskip

Our article is organized as follows. In
Section~\ref{sec:model}, we present the singularly perturbed ODE that
is considered here as a model problem, and state some bounds on its
solution (The proof of these bounds is postponed until
Appendix~\ref{sec:appendix}). Subsequently, in Section~\ref{sec:parareal}, we
introduce two micro-macro parareal algorithms.  The coupling between the
microscopic and macroscopic levels of description is done using a
\emph{restriction} operator (to go from the microscopic
to the macroscopic level), and either a \emph{lifting} (Algorithm 1) or a
\emph{matching} (Algorithm 2) operator (to go from the macroscopic
to the microscopic level). This coupling ensures that the numerical solution
remains consistent across both levels of description (see
Section~\ref{sec:parareal}).
The two algorithms we introduce in Section~\ref{sec:parareal-alg} only
differ in how the levels of description are coupled to each
other. Algorithm~\ref{algorithm-lifting}
will turn out to be inaccurate, whereas 
Algorithm~\ref{algorithm-matching} is extremely accurate. For the sake
of comparison, we discuss in Section~\ref{sec:compare_scheme} the scheme
proposed in~\cite{BBK,maday41parareal}, that we denote here Algorithm~3.
Section~\ref{sec:analysis} contains a detailed numerical
analysis of these three algorithms, when applied to the
linear model problem presented in Section~\ref{sec:model}, and when the
dynamics at both microscopic and macroscopic levels of description are
exactly integrated.
This setting enlightens the effect of how the two levels of description
are coupled on the convergence of the algorithms. We show how the
modeling error of the approximate macroscopic model affects the
accuracy. In particular, the micro-macro parareal algorithm we introduce
is a precise approximation of
the full microscopic solution only if special care is taken during the
coupling of the microscopic and macroscopic levels of description, as is
done in Algorithm~\ref{algorithm-matching}.
The analysis is illustrated by numerical experiments in
Section~\ref{sec:num}, where, in addition, we  
numerically investigate the effect of time
discretization. Some numerical results on nonlinear problems are
presented in Section~\ref{sec:nonlin}. We observe there the same good
properties of Algorithm~\ref{algorithm-matching} as on linear problems.
We conclude in Section~\ref{sec:concl} with some final remarks and a
discussion of possible future research. 

\section{Model problem\label{sec:model}}

In this section, we describe the microscopic model problem considered in
this work, as well as its macroscopic limit.

Consider the dynamics
\begin{equation}
\label{eq:lin_micro}
\dot{x} = \alpha x + p^T y, 
\quad
\dot{y} = \dfrac{1}{\epsilon}\left(q x - A y\right),
\end{equation}
where $x \in \RR$ and $y\in\RR^{d-1}$ are the state variables, and
$\alpha \in \RR$, $p \in \RR^{d-1}$, $q \in \RR^{d-1}$ and $A \in
\RR^{(d-1)\times(d-1)}$ are parameters. This dynamics models the
evolution of a system 
described by the state variable $u =(x,y) \in \RR^d$, where the slow and
fast components are $x$ and $y$, respectively. We denote the initial condition by
$u(0)=(x(0),y(0))=(x_0,y_0)=u_0$. The dynamics can
be compactly written as
\begin{equation}
\label{eq:lin_micro_short}
\dot{u}=B^\epsilon u,
\end{equation}
where
$$ 
B^\epsilon = \begin{bmatrix}
\alpha & p^T \\
q/\epsilon & -A/\epsilon
\end{bmatrix}.
$$
In the following, we assume that the fast component of the system has a
simple dissipative structure:
\begin{equation}
\label{ass:expo-stable}
\begin{array}{c}
\text{We assume $A$ to be a matrix with eigenvalues
$\lambda_i \in \CC$ ($1 \le i \le d-1$)}
\\ \noalign{\vskip 3pt}
\text{satisfying $\Re (\lambda_i) \geq \lambda_-$ for any $1 \leq i \leq
d-1$, for some $\lambda_->0$.}
\end{array}
\end{equation}
Under this assumption, for each fixed value $x=x^\star$ of the slow component, the dynamics of~$y$, obeying the equation 
\[
\dot{y} = \dfrac{1}{\epsilon}\left(q x^\star - A y\right),
\]
satisfies
$$
\lim_{t\to\infty} y(t)= \left( A^{-1} q \right) x^\star.
$$
The dynamics of the fast component $y$, for fixed slow component
$x=x^\star$, is thus exponentially stable for all $x^\star$. It is then
known (see Lemma~\ref{lem:slaving} below and, for example, \cite{Pavliotis2008} and references therein) that, in
the limit $\eps$  goes to zero, the solution $x(t)$ to~\eqref{eq:lin_micro}
converges, on finite time intervals, to the solution $X(t)$ of
\begin{equation}\label{eq:lin_macro}
\dot{X}=\lambda X, \qquad X(0)=x_0, \qquad 
\lambda := \alpha + p^T A^{-1} q.
\end{equation}
Comparing~\eqref{eq:lin_micro} with~\eqref{eq:lin_macro}, one can see
that the microscopic time-scale (namely the typical time-step required
to integrate the full microscopic dynamics~\eqref{eq:lin_micro}) is of
the order of $\epsilon$, whereas the macroscopic time-scale (namely the
typical time-step required to integrate the approximate macroscopic
dynamics~\eqref{eq:lin_macro}) is independent of $\epsilon$. 

\smallskip

\begin{rem}
The asymptotic result that we mentioned above on the
system~\eqref{eq:lin_micro} holds for more general cases. For instance,
consider the dynamics 
\begin{equation}\label{eq:micro}
\dot{x} = f(x,y), 
\quad
\dot{y} = \dfrac{1}{\epsilon}\left(\eta(x) - A y\right),
\end{equation}
with again $x \in \RR$, $y \in \RR^{d-1}$ and $A \in
\RR^{(d-1)\times(d-1)}$, and where $f : \RR \times \RR^{d-1} \to \RR$
and $\eta:\RR \to \RR^{d-1}$ are two given, possibly nonlinear
functions. Under Assumption~\eqref{ass:expo-stable}, 
the solution $x(t)$ to~\eqref{eq:micro} converges to $X(t)$, solution to 
$$
\dot{X}=F(X), \qquad X(0)=x_0, \qquad
F(X) = f(X,A^{-1}\eta(X)).
$$
This result can also be extended to more general nonlinear cases~\cite{Pavliotis2008}. 
\end{rem}

\smallskip

For future reference, we introduce the exact time evolution operators,
$$
u(t^*+\Delta t) = \Phi_{\Delta t}\left(u(t^*)\right),
\quad
X(t^*+\Delta t) = \rho_{\Delta t}\left(X(t^*)\right),
$$
corresponding to~\eqref{eq:lin_micro_short} and~\eqref{eq:lin_macro},
respectively. These equations are linear, hence the operators
$\Phi_{\Delta t}$ and $\rho_{\Delta t}$ are linear:
\begin{align}
%u(t^*+\Delta t) &= \Phi_{\Delta t} u(t^*), \qquad 
\Phi_{\Delta t} &= \exp(B^\epsilon\Delta t) \in \RR^{d \times d},
\label{eq:lin_micro_ex}\\
%X(t^*+\Delta t) &= \rho_{\Delta t} X(t^*), \qquad 
\rho_{\Delta t} &= \exp(\lambda  \Delta t) \in \RR.
\label{eq:lin_macro_ex}
\end{align}

We now state some bounds on the solutions of~\eqref{eq:lin_micro}, that
will be useful in Section~\ref{sec:analysis}, when proving error bounds
on the algorithms we propose. 

\smallskip

\begin{lem}
\label{lem:slaving}
Consider the linear system~\eqref{eq:lin_micro} over the time range
$[0,T]$, with the initial condition 
$x(0)=x_0$, $y(0)=y_0$. Introduce $z(t)=y(t)-A^{-1}q\,x(t) \in
\RR^{d-1}$ and $z_0 = z(0)$. 
Under Assumption~\eqref{ass:expo-stable}, there exist
$\epsilon_0 \in (0,1)$ and $C>0$, that both only
depend on $A$, $q$, $p$, $\alpha$ and $T$, such that, for all
$\epsilon < \epsilon_0$, 
\begin{align}
\sup_{t \in [0,T]} |x(t)-x_0 \exp(\lambda t)| &\le C\epsilon (|x_0|+\|z_0\|), 
\label{eq:x-est-lem} 
\\
\sup_{t \in [0,T]} \|z(t) - \exp\left(-A t/\epsilon \right) z_0 \| & \le 
C \eps \left( \left| x_0 \right| + \left\| z_0 \right\|\right).
\label{eq:z-est-lem-bl}
\end{align}
Set 
\begin{equation}
\label{eq:def_t_BL}
t^{\rm BL}_\epsilon = \frac{2\epsilon}{\lambda_-} \ln(1/\epsilon).
\end{equation} 
Then, for all
$\epsilon<\epsilon_0$, we have 
\begin{equation}
\sup_{t \in [t^{\rm BL}_\epsilon,T]} \|z(t)\| \le C \epsilon \left( \left| x_0
  \right| + \left\| z_0 \right\|\right). 
\label{eq:z-est-lem}
\end{equation}
Hence, up to a boundary layer of size $t^{\rm BL}_\eps$, $z(t)$ is of order
$\epsilon$, and the state $u(t)$ of the system is at a distance of the order
of $\epsilon$ of the manifold
\begin{equation}
\label{eq:def_slow_manifold} 
\Sigma := \left\{ u=(x,y) \in \RR^d; \ y = A^{-1} q\,x \right\}.
\end{equation} 
\end{lem}

We call the manifold $\Sigma$ the \emph{slow manifold}.
Note that the bound~\eqref{eq:z-est-lem} is sharp in the sense that,
after the initial time boundary layer, $z(t)$ is of order $\epsilon$ and
not smaller. This can be checked for example on the analytically
solvable system $\dot{x} = -x$, $\dot{y} = (x-y)/\epsilon$.

An important consequence of the above lemma is that the microscopic
solution $u(t)=(x(t),y(t))$ remains bounded, independently of
$\epsilon$, on the time range $[0,T]$.
The following result, which will be used repeatedly in the sequel,
follows immediately from Lemma~\ref{lem:slaving}:

\smallskip

\begin{cor}
\label{cor:slaving}
Consider the linear system~\eqref{eq:lin_micro} over the time range
$[0,T]$, with initial condition $x(0)=x_0$, $y(0)=y_0$. Under
Assumption~\eqref{ass:expo-stable}, there exist
$\epsilon_0 \in (0,1)$ and $C>0$, that both only
depend on $A$, $q$, $p$, $\alpha$ and $T$, such that, for all
$\epsilon < \epsilon_0$, we have 
\begin{align}
\sup_{t \in [0,T]} |x(t)| \le C \left(|x_0| + \epsilon \|y_0\|\right), 
\label{eq:x-est-cor} 
\\
\sup_{t \in [t^{\rm BL}_\epsilon,T]} \|y(t)\| \le 
C \left(|x_0| + \epsilon \|y_0\|\right),
\label{eq:y-est-cor}
\end{align}
where the size $t^{\rm BL}_\epsilon$ of the boundary layer is defined
by~\eqref{eq:def_t_BL}. 
\end{cor}

\smallskip

The proofs of these standard results are postponed until
Appendix~\ref{sec:appendix}. In view of~\eqref{eq:x-est-lem}, we see
that, in the limit when $\epsilon$ goes to zero, the macroscopic
dynamics~\eqref{eq:lin_macro} is exact. The aim of the algorithms we
investigate below is to use these macroscopic dynamics to speed up the
computation of the solution of the original model~\eqref{eq:lin_micro},
for a fixed small but non-zero value of $\epsilon$.

\section{Micro-macro parareal algorithms\label{sec:parareal}}

In this section, we describe two micro-macro parareal
algorithms. As 
will become clear from the analysis in the forthcoming sections, the 
first one based on a lifting operator is inaccurate, whereas the
second one based on a matching operator is
extremely accurate. Both are 
generalizations of the parareal 
algorithm proposed in~\cite{Lions2001}. Our formulation follows most
closely the description in~\cite{baffico2002parallel}. We first
introduce the necessary notation in Section~\ref{sec:parareal-notation},
and we subsequently outline both algorithms in
Section~\ref{sec:parareal-alg}. For the sake of comparison, we also
discuss in Section~\ref{sec:compare_scheme} the scheme proposed
in~\cite{BBK,maday41parareal}.
Let us emphasize that the two algorithms we introduce are not restricted
to the linear system~\eqref{eq:lin_micro}, and apply to any system
of the form 
$$
\dot{x} = f(x,y), 
\quad
\dot{y} = \dfrac{1}{\epsilon}g(x,y),
$$
where $x \in \RR^s$ is a slow component ($s \in \NN^\star$), $y \in
\RR^m$ is a fast component ($m \in \NN^\star$), and where the associated
macroscopic dynamics (obtained in 
the limit of infinite time scale separation between the slow and the
fast components, namely in the limit when $\epsilon$ goes to zero) reads
$\dot{X} = F(X)$.

\subsection{Notation\label{sec:parareal-notation}}

%% CHANGED: made consistent use of microscopic/macroscopic. Only propagators are now fine-scale/coarse

We introduce a time discretization $(t_n)_{n=0}^{N}$, with $t_n=n\Delta
t$ and $N\Delta t=T$. Let $u^n=(x^n,y^n) \approx u(t_n)=(x(t_n),y(t_n))$
be the numerical approximation of the solution of the
microscopic model~\eqref{eq:lin_micro}, and let $X^n \approx
X(t_n)$ be that of the solution to the macroscopic
model~\eqref{eq:lin_macro}. 

\medskip

\paragraph{Fine-scale and coarse propagators}

The micro-macro parareal algorithm makes use of two propagators. First,
we need a \emph{fine-scale propagator}, that advances the microscopic
model~\eqref{eq:lin_micro} over a time-range $\Delta t$:  
\begin{equation}\label{eq:fine-propagator}
u^{n+1}=\F_{\Delta t}(u^n).
\end{equation}
To perform this, we may consider that we have at hand the {\em
  exact} propagator of the equation~\eqref{eq:lin_micro}, in which case
$\F_{\Delta t} \equiv \Phi_{\Delta t}$, where $\Phi_{\Delta t}$ is
defined by~\eqref{eq:lin_micro_ex}. Alternatively, we may resort to a
numerical integration of the dynamics~\eqref{eq:lin_micro} (using
for example forward or backward Euler discretizations) over the time range
$\Delta t$, using several steps of size $\delta t$. Typically, in the
context of a system like~\eqref{eq:lin_micro}, one would need $\delta
t$ to be of the order of $\epsilon$ to obtain accurate results. 

Second, we need a \emph{coarse propagator} for the macroscopic
model~\eqref{eq:lin_macro}, 
\begin{equation}\label{eq:coarse-propagator}
X^{n+1}=\C_{\Delta t}(X^n),
\end{equation}
where again we may assume that we can exactly
integrate~\eqref{eq:lin_macro} and hence choose $\C_{\Delta t} \equiv
\rho_{\Delta t}$, see~\eqref{eq:lin_macro_ex}. Alternatively, one may
resort to a numerical integration of the dynamics~\eqref{eq:lin_macro},
for which we can use a time-step independent of $\epsilon$ to obtain
accurate results. 

\medskip

\paragraph{Restriction, lifting and matching operators}

The parareal algorithm iteratively uses the fine-scale and the coarse
propagators. In this work, these two
propagators correspond to {\em different descriptions} of the system,
either microscopic (using $u \in \RR^d$) or macroscopic (using $X \in
\RR$). We thus need a way to go from one description to the other, as we
discuss now.  

We first introduce the \emph{restriction} operator
\begin{equation*}
\R:
\left\{
\begin{array}{rcl}
 \RR^d & \to & \RR 
\\
u=(x,y) &\mapsto& x,
\end{array}
\right.
\end{equation*} 
which maps a microscopic state to the corresponding macroscopic
state. For notational convenience, we also introduce the complement of
the restriction operator, 
\begin{equation*}
\R^\perp:
\left\{
\begin{array}{rcl}
 \RR^d & \to & \RR^{d-1}
\\
u=(x,y) &\mapsto& y,
\end{array}
\right.
\end{equation*} 
such that we can write $u=(x,y)=(\R u, \R^\perp u)$.

Conversely, we will also need to reconstruct a microscopic state from a
given macroscopic state. In contrast to the restriction operator, there
is no unique obvious way to define this operator. We
introduce two such operators, a \emph{lifting} operator and a
\emph{matching} operator. 

\medskip

\begin{defn}
A \emph{lifting operator} $\L$ is an operator
\begin{equation*}
\L:
\left\{
\begin{array}{rcl}
 \RR&\to&\RR^d
\\
X &\mapsto& u = \L(X)
\end{array}
\right.
\end{equation*}
that creates a microscopic state that is uniquely
determined by a given macroscopic state and satisfies the consistency property
\begin{equation}
\label{eq:consistance_R}
\R \circ \L=\Id. 
\end{equation}
%\hfill $\diamond$
\end{defn}

A possible choice is to take $\L(X)$ such that 
\begin{equation}
\label{eq:R_Sigma}
\R (\L(X))=X \quad \text{and} \quad \L(X) \in \Sigma,
\end{equation}
where $\Sigma$ is the slow manifold associated to the multiscale
problem.

In connection with the system~\eqref{eq:lin_micro}, an example (and this
is the choice we make in this work) is to choose 
\begin{equation}
\label{eq:choix_L}
\L(X) = (X, (A^{-1}q) X).
\end{equation}
This choice indeed satisfies~\eqref{eq:consistance_R}
and~\eqref{eq:R_Sigma}, where the slow manifold $\Sigma$ of the
system~\eqref{eq:lin_micro} is defined by~\eqref{eq:def_slow_manifold}. 

\smallskip

\begin{rem}
Other lifting operators can be introduced, using for example the
constrained runs algorithm~\cite{Gear2005}. As soon as the lifting
operator $\L$ is 
specified, $u$ is uniquely determined by $X$: the lifting operator
enforces a closure approximation on the microscopic state. 
%\hfill $\diamond$
\end{rem}

\smallskip

Alternatively, one may reconstruct a microscopic state using a 
\emph{matching operator}. 

\smallskip

\begin{defn}
A \emph{matching operator} is an operator
\begin{equation*}
\P: 
\left\{
\begin{array}{rcl}
\RR \times \RR^d &\to& \RR^d
\\
(X,v) &\mapsto& \P_X(v),
\end{array}
\right.
\end{equation*}
that satisfies
\begin{equation}
\label{eq:consistance}
X = \left(\R \circ \P\right)(X,v) \text{ for any $v \in \RR^d$ and $X
  \in \RR$},
\end{equation}
and
$$
\forall u \in \RR^d \text{ such that } \R(u)=X, \quad \P_X(u) = u,
$$
or, equivalently,
\begin{equation}
\label{eq:P_fundamental}
\forall u \in \RR^d, \quad \P(\R(u),u) = u.
\end{equation}
In contrast with a lifting operator, a matching operator requires a
microscopic state as an input, and not only a macroscopic state. 
%\hfill $\diamond$
\end{defn}

\medskip

The consistency property~\eqref{eq:consistance} may be seen as the
equivalent for $\P$ of the property~\eqref{eq:consistance_R} for $\L$.
We also note that, in view of~\eqref{eq:P_fundamental}, a
microscopic state $u$ 
which is already consistent with the macroscopic value $X$ is unaltered
by the operator $\P_X$. Combining~\eqref{eq:consistance}
and~\eqref{eq:P_fundamental}, we observe that $\P_X \circ \P_X = \P_X$:
the operator $\P_X: \RR^d \to \RR^d$ is thus a projection operator onto
microscopic states $u \in \RR^d$ that satisfy $\R(u)=X$. 
One may thus think of $\P_X$ as a projection operator that projects a
microscopic state $v$ to a microscopic state $u = \P_X(v)$, such that
$\R(u)=X$ and $u$ is as close to $v$ as possible, in a sense to be made
precise for the problem at hand. 
 
In the following, we require in addition the following continuity
property on $\P$: there exists $C>0$ such that, for all $X \in
  \RR$, $Y \in \RR$, $u \in \RR^d$ and $v \in \RR^d$,
\begin{equation}
\label{eq:P_b}
\left\| \P(X,u) - \P(Y,v) \right\| \leq C \Big[
\left\| u - v \right\| + \left| X-Y \right| \Big].
\end{equation}

For the {\em analysis} of the algorithms described below, we only require $\P$
to satisfy~\eqref{eq:consistance},~\eqref{eq:P_fundamental}
and~\eqref{eq:P_b}, and do not make any additional assumptions (see
Section~\ref{sec:analysis}). For the {\em numerical experiments}
reported on in Section~\ref{sec:num}, we choose, in the context of the
system~\eqref{eq:lin_micro}, 
\begin{equation}\label{eq:projection-good}
\P_X(v) := (X,\R^\perp v),
\end{equation}
which consists in keeping the fast variables from $v$, while imposing
the slow variable to be equal to $X$. This choice fulfills all the
above conditions~\eqref{eq:consistance},~\eqref{eq:P_fundamental}
and~\eqref{eq:P_b}.

\smallskip

\begin{rem}
The term \emph{matching operator} has been chosen in reminiscence of the
term ``moment matching'' that is commonly used in the Monte Carlo
community, see e.g.~\cite{caflisch1998monte}.
%\hfill $\diamond$
\end{rem}

\subsection{Algorithms 1 and 2\label{sec:parareal-alg}}

The parareal algorithm iteratively constructs approximations on the
whole time domain $[0,T]$. We denote by $u_k^n$, resp.~$X_k^n$, the approximate
microscopic, resp.~macroscopic, solution at time $t_n$, obtained at
the $k$-th parareal iteration.  

The first algorithm we consider is the following.

\smallskip

\begin{alg}
\label{algorithm-lifting}
Let $u(0)=u_0$ be the initial condition. 
\begin{enumerate}
\item Initialization: 
\begin{enumerate}[a)]
\item Compute $\left\{ X_0^n \right\}_{0 \leq n \leq N}$
  sequentially by using the coarse propagator:
$$
X^0_0 = \R(u_0),\qquad X_0^{n+1}=\C_{\Delta t}(X_0^n).
$$
\item Lift the macroscopic approximation to the microscopic level:
$$
u_0^0 = u_0 
\quad \text{and, for all $1 \leq n \leq N$,} \quad
u_0^n = \L(X_0^n).
$$
\end{enumerate}
\item Assume that, for some $k \geq 0$, the sequences $\left\{ u^{n}_{k}
  \right\}_{0 \leq n \leq N}$ and $\left\{ X^{n}_{k} \right\}_{0 \leq n
    \leq N}$ are known. Compute these sequences at the iteration $k+1$
  by the following steps:
\begin{enumerate}[a)]
\item For all $0 \leq n \leq N-1$, compute (in parallel) using the
  coarse and the fine-scale propagators
\begin{equation}
\overline{X}_k^{n+1}=\C_{\Delta t}(X_k^n), 
\quad
\overline{u}_k^{n+1}=\F_{\Delta t}(u_k^n).
\label{eq:micro_prop}
\end{equation}
\item For all $0 \leq n \leq N-1$, evaluate the jumps (the difference
  between the two propagated values) \emph{at the macroscopic level}: 
\begin{equation}\label{eq:parareal-jumps}
J_k^{n+1}= \R(\overline{u}^{n+1}_k)-\overline{X}^{n+1}_k.
\end{equation}
\item Compute $\left\{ X^{n}_{k+1} \right\}_{0 \leq n \leq N}$
  sequentially by
\begin{equation}\label{eq:parareal-propagate}
X^0_{k+1} = \R(u_0), \qquad X_{k+1}^{n+1} = \C_{\Delta t}(X^n_{k+1})+J^{n+1}_k.
\end{equation}
\item Compute $\left\{ u^{n+1}_{k+1} \right\}_{0 \leq n \leq N-1}$ by
  lifting the macroscopic solution:
\begin{equation}
\label{eq:parareal-lifting}
u_{k+1}^0 = u_0 
\quad \text{and, for all $0 \leq n \leq N-1$,} \quad
u_{k+1}^{n+1}=\L(X_{k+1}^{n+1}).
\end{equation}
%\hfill $\diamond$
\end{enumerate}
\end{enumerate}
\end{alg}

We can recast the above algorithm as
\begin{equation}
\label{eq:para_scheme_lift}
u^{n+1}_{k+1}
=
\L\Big(
\C_{\Delta t}\left(\R\left(u^n_{k+1}\right)\right)
+
\R\left(\F_{\Delta t}\left(u_k^n\right)\right)
-
\C_{\Delta t}\left(\R\left(u^n_k\right)\right)
\Big).
\end{equation}
Notice that this cannot be recast in the form of the original parareal
algorithm~\eqref{eq:scheme-parareal}. The above algorithm uses the
following paradigm: each time we need to 
reconstruct a full microscopic solution $u$ from a given macroscopic
state $X$, we use the lifting operator $\L$.
For example, for the system~\eqref{eq:lin_micro} and $\L$ given
by~\eqref{eq:choix_L}, this amounts to creating a microscopic state
exactly on the slow manifold~\eqref{eq:def_slow_manifold}. 

We will see in the sequel that this algorithm leads to disappointing
results. In particular, Algorithm~\ref{algorithm-lifting} does not
retain one of the properties of the parareal algorithm as originally
proposed in~\cite{Lions2001}, namely that the numerical trajectory is
exact on the first $k$ subintervals in time after $k$ iterations of
the parareal algorithm. 

\medskip

A much better algorithm is the following:

\smallskip

\begin{alg}
\label{algorithm-matching}
Let $u(0)=u_0$ be the initial condition. 
\begin{enumerate}
\item Initialization: 
proceed as in Step 1 of Algorithm~\ref{algorithm-lifting}.
\item Assume that, for some $k \geq 0$, the sequences $\left\{ u^{n}_{k}
  \right\}_{0 \leq n \leq N}$ and $\left\{ X^{n}_{k} \right\}_{0 \leq n
    \leq N}$ are known. To compute these sequences at the iteration
  $k+1$,
\begin{itemize}
\item Proceed as in Steps 2a, 2b and 2c of
  Algorithm~\ref{algorithm-lifting}.
\item Compute $\left\{ u^{n+1}_{k+1} \right\}_{0 \leq n \leq N-1}$ by
  matching the result of the local microscopic computation,
  $\overline{u}_k^{n+1}$, on the corrected macroscopic state $X_{k+1}^{n+1}$:
\begin{equation}
\label{eq:parareal-projection}
u_{k+1}^0 = u_0 
\quad \text{and, for all $0 \leq n \leq N-1$,} \quad
u_{k+1}^{n+1}=\P(X_{k+1}^{n+1},\overline{u}_k^{n+1}).
\end{equation}
%\hfill $\diamond$
\end{itemize}
\end{enumerate}
\end{alg}
The only difference between Algorithms~\ref{algorithm-lifting}
and~\ref{algorithm-matching} is how we reconstruct the microscopic
solution $u_{k+1}^{n+1}$. In Algorithm~\ref{algorithm-lifting}, we
simply choose $u_{k+1}^{n+1}$ on the slow manifold defined by $X_{k+1}^{n+1}$ 
(see~\eqref{eq:parareal-lifting}). In
Algorithm~\ref{algorithm-matching}, we use the quantity
$\overline{u}_k^{n+1}$, which is the end point of a microscopic trajectory
between times $n \Delta t$ and $(n+1) \Delta t$, and match this state
onto the corrected macroscopic state $X_{k+1}^{n+1}$, obtained at the
latest parareal iteration. 

At the initial iteration $k=0$, since no microscopic computation has been
done, we cannot use the matching operator $\P$ to reconstruct a fine-scale
solution. We thus resort to the lifting operator $\L$.

Algorithm~\ref{algorithm-matching} can be recast as
\begin{equation}
\label{eq:para_scheme_match}
u^{n+1}_{k+1}
=
\P\Big(
\C_{\Delta t}\left(\R\left(u^n_{k+1}\right)\right)
+\R\left(\F_{\Delta t}\left(u_k^n\right)\right)
-\C_{\Delta t}\left(\R\left(u^n_k\right)\right),\F_{\Delta t}(u^n_k)\Big),
\end{equation}
which is to be compared with the original parareal
algorithm~\eqref{eq:scheme-parareal} and~\eqref{eq:para_scheme_lift} for
Algorithm~\ref{algorithm-lifting}. For the linear
system~\eqref{eq:lin_micro} and the choice~\eqref{eq:projection-good} of
matching operator, the
equation~\eqref{eq:para_scheme_match} can be further simplified to
\begin{equation}
\label{eq:para_scheme_match_expl}
u^{n+1}_{k+1}
=
\F_{\Delta t}(u^n_k)
+(1,0)^T \Big(
\C_{\Delta t}\left(\R\left(u^n_{k+1}\right)\right)
-
\C_{\Delta t}\left(\R\left(u^n_k\right)\right)\Big).
\end{equation}
This is exactly~\eqref{eq:scheme-parareal} with $\F_{\Delta t}$ as the
fine propagator and $(1,0)^T \C_{\Delta t} \R$ as the coarse propagator.

Note that, in view of~\eqref{eq:consistance_R}
and~\eqref{eq:parareal-lifting} (respectively~\eqref{eq:consistance}
and~\eqref{eq:parareal-projection}), the trajectories computed using
Algorithm~\ref{algorithm-lifting} (respectively
Algorithm~\ref{algorithm-matching}) satisfy
\begin{equation}
\label{eq:consis}
\forall k \geq 0, \quad \forall n \geq 0, \quad
X_k^n = \R(u_k^n).
\end{equation} 
At any parareal iteration $k$, the macroscopic trajectory is
{\em consistent} with the microscopic trajectory. 

\subsection{Comparison of Algorithms~\ref{algorithm-lifting} and~\ref{algorithm-matching} with that of~\cite{BBK,maday41parareal}\label{sec:compare_scheme}} 

As underlined in the introduction, a micro-macro version of the parareal
algorithm has already been proposed in~\cite{maday41parareal,BBK}. In
these works, the coarse propagator is an integrator of a reduced (DAE)
model that contains {\em all} degrees of freedom in the system (both the
fast and slow ones), in contrast to our algorithms, where the coarse
propagator is an integrator for the effective dynamics of the slow
degrees of freedom.

For the model problem~\eqref{eq:lin_micro}, the reduced DAE considered
in~\cite{maday41parareal,BBK} takes the form 
\begin{equation}
\label{eq:dae}
\dot{x}= \alpha x + p^Ty, \qquad Ay=qx.	
\end{equation} 
The coarse propagator of~\cite{maday41parareal,BBK} is an integrator
$\G_{\Delta t}$ of~\eqref{eq:dae}. This coarse integrator is combined
with a fine-scale integrator $\F_{\Delta t}$ of~\eqref{eq:lin_micro} in
the parareal fashion, following~\eqref{eq:scheme-parareal}.

The obtained scheme, that we denote here Algorithm~3, differs from our
Algorithm~\ref{algorithm-matching} 
in its treatment of the fast degrees of freedom. To show this, we note
that, specifically for the model problem~\eqref{eq:lin_micro}, an exact
propagator for~\eqref{eq:dae} can be obtained by first
solving~\eqref{eq:lin_macro} exactly, and second solving the algebraic
equation for $y$. Hence, we have 
\begin{equation}
\label{eq:coarse_maday}
\G_{\Delta t}(u) 
= 
\L\circ\rho_{\Delta t}\circ \R \; u
=
\L\circ\C_{\Delta t}\circ \R \; u
\end{equation}
where, we recall, $\C_{\Delta t}$ is the coarse propagator used in
Algorithms~\ref{algorithm-lifting} and~\ref{algorithm-matching}.
Using~\eqref{eq:scheme-parareal}, we write Algorithm~3 as follows:
\begin{align}
\label{eq:para_scheme_maday}
u^{n+1}_{k+1} &= 
\F_{\Delta t}(u^n_k)
+ \G_{\Delta t}\left(u^n_{k+1}\right) - 
\G_{\Delta t}\left(u^n_k\right)
\nonumber
\\
& =
\F_{\Delta t}(u^n_k)
+\L\Big(
\C_{\Delta t}\left(\R\left(u^n_{k+1}\right)\right) - 
\C_{\Delta t}\left(\R\left(u^n_k\right)\right)\Big),
\end{align}
which can be compared with~\eqref{eq:para_scheme_match_expl} and
with~\eqref{eq:para_scheme_lift}. Notice in particular that
Algorithm~\ref{algorithm-matching} differs from Algorithm 3 in the
choice of the coarse propagator.

\medskip

The three algorithms only differ in how the microscopic and macroscopic
levels of description are coupled in the parareal iterations. These
differences, however, have implications on (i) the computational
complexity of the methods; (ii) the way they generalize to more complex
multiscale systems; and (iii) the convergence behavior. The convergence
properties of the three algorithms are analyzed in
Section~\ref{sec:analysis}. We here briefly comment on the other two
aspects. 
First, the computational complexity of the coarse propagator in
Algorithm~3 is significantly higher than that of
Algorithms~\ref{algorithm-lifting} and~\ref{algorithm-matching}, due to
the presence of the fast degrees of freedom, which requires solving a
large linear system in addition to the time-stepping of the slow degrees
of freedom. 

% \fred{la fin de cette section (jusqu'au debut de la
%   section~\ref{sec:analysis}) me semble encore \`a revoir. J'y ai
%   travaill\'e, mais relisez/amendez dans le detail}
Second, in more complex situations, for instance when the microscopic
and macroscopic systems are nonlinear, Algorithm~3 may require the use
of a time integrator for DAEs. Although many such integrators exist,
they are usually implicit, and less convenient than ODE solvers. In
those cases, Algorithms~\ref{algorithm-lifting}
and~\ref{algorithm-matching} only require a reasonable model to
propagate the macroscopic variables. 
In both cases, one may resort to computational multiscale
methods that approximate the evolution of the approximate macroscopic
model by using short microscopic simulations. The coarse
propagator required in Algorithms~\ref{algorithm-lifting}
and~\ref{algorithm-matching} can be replaced by a coarse projective
integration
approach~\cite{KevrGearHymKevrRunTheo03,Kevrekidis:2009p7484}. 
The coarse propagator for the DAE system required in Algorithm~3 can be
replaced by a projective integration
method~\cite{gear2006towards,gear2003projective}.  
Remark that the
computational cost of both methods is not identical: projective
integration requires a computational cost of $O(\log(1/\epsilon))$,
whereas the computational cost of coarse projective integration is
independent of $\epsilon$. This shows again that Algorithms~\ref{algorithm-lifting}
and~\ref{algorithm-matching} are cheaper to implement than Algorithm~3.
% \giovanni{Next section is very colloquial.} \tony{Je ne suis pas sur
%   de voir l'apport de ce qui suit par rapport a ce qui precede ?}
% A third difference is the way the lifting operator is used.  In
% Algorithm~\ref{algorithm-matching}, it is not used, except if the coarse
% propagator is coarse projective integration; then it is used at that
% level.
% \fred{je ne comprends pas la phrase ci-dessus. On utilise le lifting \`a
%   l'iteration 0, non? ou bien on sous-entend que de toute facon, tous
%   les coups sont permis \`a l'iteration 0, car on corrigera derriere, et
% on discute donc des iterations $k \geq 1$???}
% In Algorithm~3, it is used during the propagation of the jumps. The
% computational cost associated with the lifting in coarse projective
% integration in Algorithm~\ref{algorithm-matching} is independent of
% $\epsilon$, whereas the computational cost of the projective integration
% in Algorithm~3 is $O(\log(1/\epsilon))$, therefore larger for small
% $\epsilon$. \fred{on se repete, il me semble} 
% Additionally, Algorithm~3 requires the lifting operator
% (and its associated computational cost) again when propagating the
% jumps. Hence, the computational cost per iteration for Algorithm~3 is
% higher than for Algorithm~\ref{algorithm-matching}.
 We will see in the next Section to what extent the higher computational cost of Algorithm~3 allows for a better accuracy. 
%\fred{fin de mes interrogations metaphysiques}
% 

\section{Analysis\label{sec:analysis}}

In this section, we analyze the convergence of the two 
micro-macro parareal algorithms introduced above on the linear model
problem~\eqref{eq:lin_micro}. We also give a detailed analysis for
Algorithm 3, introduced in~\cite{BBK,maday41parareal}.
To keep the analysis simple, we focus on 
the error due to the fact that the models are different at the macroscopic and
microscopic levels. We thus
track the dependency of the error bounds on the parameter $\epsilon$,
and consider, at both levels, the exact
propagators~\eqref{eq:lin_micro_ex} and~\eqref{eq:lin_macro_ex}.
Thus, the fine-scale and coarse propagators
in~\eqref{eq:fine-propagator} and~\eqref{eq:coarse-propagator} are given
by 
$$
\F_{\Delta t}(u)=\Phi_{\Delta t}u, \qquad \C_{\Delta t}(X)=\rho_{\Delta t} X,
$$
for a fixed $\Delta t$, which is chosen typically much larger than
$\epsilon$ (so that $\Delta t$ is a macroscopic time-scale). We recall
that the lifting operator $\L$ is defined by~\eqref{eq:choix_L}, and
that we work with a matching operator $\P$
satisfying~\eqref{eq:consistance},~\eqref{eq:P_fundamental}
and~\eqref{eq:P_b}. 

We first derive an error recursion formula
in Section~\ref{sec:recursion}. Using this formula, we derive a
sharp error bound on the trajectories computed by
Algorithm~\ref{algorithm-lifting}, where the microscopic state is
reconstructed using the lifting operator $\L$ (see
Section~\ref{sec:conv-lifting}). We next turn to
Algorithm~\ref{algorithm-matching}, where the microscopic state is
reconstructed using a matching operator~$\P$. We first show
that, at a given parareal iteration $k$, the computed trajectories (both
at the macro and the micro scales) are exact up to the time $k \Delta
t$ (see Section~\ref{sec:exact}), reproducing thereby a 
property of the standard parareal algorithm. 
We subsequently derive a sharp error
bound in terms of $\epsilon$,
showing that, at iteration $k$, Algorithm~\ref{algorithm-matching}
converges to the exact solution of the full microscopic system with an
error of the order of $\epsilon^{k/2}$ (see 
Section~\ref{sec:conv-proj} for precise statements). These two
properties (exactness of the trajectories up to time $k \Delta t$ after
$k$ iterations, and improvement of the convergence rate to the exact
solution as $k$ increases) are
not satisfied for Algorithm~\ref{algorithm-lifting}. 
We eventually consider Algorithm~3. Being based
on~\eqref{eq:scheme-parareal}, this algorithm automatically satisfies
the local exactness property. We then prove a sharp error
bound in terms of $\epsilon$, showing, in agreement
with~\cite{maday41parareal}, that, at iteration $k$, Algorithm~3
converges to the exact solution of the full microscopic system with an 
error of the order of $\epsilon^k$ (see Section~\ref{sec:anal3} for
precise statements). 

The analysis below closely follows that of~\cite{Lions2001}, but is
significantly extended. We explicitly relate to the case considered
in~\cite{Lions2001} when appropriate. 

Before proceeding, we introduce two notions of error:

\smallskip

\begin{defn}[Microscopic error] 
\label{def:micro-error}
Let $u(t_n)$ be the exact microscopic solution of~\eqref{eq:lin_micro_short} at time $t_n=n\Delta t$,
and let $u^n_k$ be the parareal microscopic solution after $k$ parareal
iterations, using Algorithm~\ref{algorithm-lifting} or~\ref{algorithm-matching}. The microscopic error 
\begin{equation}
\label{eq:def_e_micro}
e^n_k = u^n_k-u(t_n)
\end{equation}
is defined as the difference of the
solutions at the microscopic level.
%\hfill $\diamond$
\end{defn}  

\smallskip

\begin{defn}[Macroscopic error] 
\label{def:macro-error}
Let $u(t_n)$ be the exact microscopic solution of~\eqref{eq:lin_micro_short}
at time $t_n=n\Delta t$, 
and let $X^n_k$ be the parareal macroscopic solution after $k$ parareal
iterations, using Algorithm~\ref{algorithm-lifting} or~\ref{algorithm-matching}. The macroscopic error
\begin{equation}
\label{eq:def_e_macro}
E^n_k = X^n_k-\R u(t_n)
\end{equation}
 is defined as the difference of the
solutions at the macroscopic level.
%\hfill $\diamond$
\end{defn}

\smallskip

Note that, in view of~\eqref{eq:consis} and using the linearity of $\R$,
we have
\begin{equation}
\label{eq:consis_e}
E^n_k = \R e^n_k.
\end{equation}
 
\subsection{Error recursion formula\label{sec:recursion}}

A first step in the analysis of the algorithms described above is the
derivation of a recursion formula for the error, which is valid for both algorithms and for any choice of the operators $\R$, $\L$ and $\P$. Starting
from~\eqref{eq:parareal-propagate} and~\eqref{eq:parareal-jumps}, we
write $X^n_{k+1}$ as a function of the microscopic and macroscopic
solutions at the parareal iteration $k$: for $n \ge 2$,
\begin{align}
X^{n}_{k+1} &= \C_{\Delta t}(X^{n-1}_{k+1})+J^n_k
\nonumber
\\
&= \rho_{\Delta t} X^{n-1}_{k+1} + \left(\R\Phi_{\Delta t}
  u_k^{n-1} - \rho_{\Delta t} X_k^{n-1}\right)
\nonumber
\\
&= \R\Phi_{\Delta t} u_k^{n-1} + \rho_{\Delta t} 
\left( X^{n-1}_{k+1} - X_k^{n-1}\right)
\nonumber
\\
&= \R\Phi_{\Delta t} u_k^{n-1} + \rho_{\Delta t} \left( \R\Phi_{\Delta
    t} u_k^{n-2} + \rho_{\Delta t} \left( X^{n-2}_{k+1} -
    X_k^{n-2}\right) - X_k^{n-1}\right)
\nonumber
\\
&= \R\Phi_{\Delta t} u_k^{n-1} + \rho_{\Delta t} \left( \R\Phi_{\Delta
    t} u_k^{n-2} - X_k^{n-1}\right) + \rho_{\Delta t}^2 \left(
  X^{n-2}_{k+1} - X_k^{n-2}\right) 
\nonumber
\\
&= \R\Phi_{\Delta t} u_k^{n-1} + \sum_{p=1}^{n-1} \rho_{\Delta t}^{p}
\left( \R\Phi_{\Delta t} u_k^{n-p-1} - X_k^{n-p} \right) 
\nonumber
\\
&= \R\Phi_{\Delta t} u_k^{n-1} + \sum_{p=1}^{n-1} \rho_{\Delta t}^{n-p}
\left( \R\Phi_{\Delta t} u_k^{p-1} - X_k^{p} \right).
\label{eq:start-exact}
\end{align}
This formula is also valid for $n=1$ using the convention $\sum_{p=1}^0
\cdot = 0$. Note that we have used the linearity of the coarse
propagator. We then obtain a recursion for the macroscopic error,
using the linearity of the fine-scale propagator:
\begin{align}
E^n_{k+1} &= X^n_{k+1}-\R\Phi_{\Delta t}^n u_0
\nonumber
\\
& 
=\R\Phi_{\Delta t} u_k^{n-1} - \R\Phi_{\Delta t}^n u_0 +
\sum_{p=1}^{n-1} \rho_{\Delta t}^{n-p} \left( \R\Phi_{\Delta t}
  u_k^{p-1} - X_k^{p} \right)
\nonumber
\\
&
=\R\Phi_{\Delta t} e_k^{n-1} + \sum_{p=1}^{n-1} \rho_{\Delta t}^{n-p}
\left( \R\Phi_{\Delta t} u_k^{p-1} -\R\Phi_{\Delta t}^p u_0 +
  \R\Phi_{\Delta t}^p u_0- X_k^{p} \right)
\nonumber
\\
&
=\R\Phi_{\Delta t} e_k^{n-1} + \sum_{p=1}^{n-1} \rho_{\Delta t}^{n-p}
\left( \R\Phi_{\Delta t} e_k^{p-1}-E_k^{p} \right)
\nonumber 
\\
&=\R\Phi_{\Delta t} e_k^{n-1} + \sum_{p=0}^{n-2} \rho_{\Delta t}^{n-p-1}
\R\Phi_{\Delta t} e_k^{p} - \sum_{p=1}^{n-1} \rho_{\Delta t}^{n-p}
E_k^{p}
\nonumber 
\\
&=\sum_{p=1}^{n-1} \rho_{\Delta t}^{n-p-1} \left( \R\Phi_{\Delta t}
  e_k^{p}-\rho_{\Delta t}E_k^{p}\right),
\label{eq:error-recursion}
\end{align}
where, in the last line, we have used the fact that the microscopic error
at $t=0$ vanishes: $e_k^{0} = 0$ for any $k$.
 
We remark that the formula \eqref{eq:error-recursion} is not closed, in
the sense that it couples the macroscopic error at parareal iteration
$k+1$ \emph{to both the macroscopic and the microscopic errors} at
parareal iteration $k$. To close the formula, and to transform it into
specific bounds on the errors, we will need to make use of specific
properties of $\Phi_{\Delta t}$ and of the lifting, matching and
restriction operators $\L$, $\P$ and $\R$. This is where the analysis of
Algorithms~\ref{algorithm-lifting} and~\ref{algorithm-matching} differ.

\smallskip

\begin{rem}
\label{rem:std_ananu}
Using~\eqref{eq:error-recursion}, it is possible to recover standard
error bounds on the parareal algorithm, when the microscopic and the
macroscopic models are linear and written at the same level of
description, using a common state variable $u\in\RR$, as
in~\cite{Lions2001} for example. In this case, we 
have $\R=\L=\P=\Id$ (there is only one model, and one level of
description), and $e_k^n = E_k^n$. The coarse and fine-scale propagators are
linear operators, denoted respectively by $\C_{\Delta t}(u) = \G_{\Delta t}(u) =
\rho^G_{\Delta t} u$ and $\F_{\Delta t}(u)=\rho^F_{\Delta t} u$. Since
$u$ is scalar, the propagators are simply multiplications by two scalars
$\rho^G_{\Delta t}$ and $\rho^F_{\Delta t}$. The
equation~\eqref{eq:error-recursion} then reads
\begin{equation}
E^n_{k+1} = \sum_{p=1}^{n-1} (\rho_{\Delta t}^G)^{n-p-1} \left( \rho_{\Delta
    t}^F -\rho_{\Delta t}^G \right) E_k^{p}.
\label{eq:error-recursion-maday}
\end{equation}
Assume, as
in the classical analysis of the parareal algorithm presented
in~\cite{Lions2001}, that the fine-scale propagator is exact, whereas the
coarse propagator is a scheme of order~$s$:
$\left| \rho_{\Delta t}^F - \rho_{\Delta t}^G \right| = O(\Delta
t^{s+1})$. We consider a range of $\Delta t$ such that $\rho_{\Delta t}^G
> 0$ (which is possible since $\rho^G_{\Delta t} = 1 + O(\Delta
t)$). Fix a time range $[0,T]$. We show that,
using~\eqref{eq:error-recursion-maday}, one
can recover the classical result of~\cite{Lions2001}: at any
parareal iteration $k$, there exists $c_k$ such that, for any $\Delta
t$, 
\begin{equation}
\label{eq:recursion1}
\sup_{0 \leq n \Delta t \leq T}
\left| E_k^n \right| \le c_k \Delta t^{s(k+1)}.
\end{equation} 
This bound is satisfied at $k=0$ since the coarse propagator is of
order $s$. Assume now that~\eqref{eq:recursion1} holds at some parareal
iteration $k$. We then deduce from~\eqref{eq:error-recursion-maday}
that, for all $n \ge 0$ such that $n \Delta t \le T$,
\begin{equation}
\label{eq:tata1}
\left|E^n_{k+1}\right| 
\le 
c_k \Delta t^{s(k+1)} \left| \rho_{\Delta t}^F - \rho_{\Delta t}^G
\right| 
\sum_{p=1}^{n-1} (\rho_{\Delta t}^G)^{n-p-1} 
\leq
C c_{k} \Delta t^{s(k+1)} \Delta t^{s+1}
\sum_{p=0}^{N-1} (\rho_{\Delta t}^G)^p. 
\end{equation}
Remark now that
$$
\sum_{p=0}^{N-1} (\rho_{\Delta t}^G)^p
=
\dfrac{(\rho_{\Delta t}^G)^N - 1}{\rho_{\Delta t}^G - 1}
=
\dfrac{(\rho_{\Delta t}^G)^N - (\rho^F_{\Delta t})^N}{\rho_{\Delta t}^G - 1}
+
\dfrac{(\rho^F_{\Delta t})^N - 1}{\rho_{\Delta t}^G - 1}.
$$
Since the fine-scale propagator is exact, we have 
$(\rho^F_{\Delta t})^N = \rho^F_{N\Delta t} = \rho^F_{T}$,
which is independent of $\Delta t$. Thus
\begin{equation}
\label{eq:tata2}
\left| \sum_{p=0}^{N-1} (\rho_{\Delta t}^G)^p \right|
\leq 
\dfrac{C \Delta t^s + C}{\left| \rho^G_{\Delta t} - 1 \right|}
\leq 
\dfrac{C}{\Delta t}.
\end{equation}
Collecting~\eqref{eq:tata1} and~\eqref{eq:tata2}, we
deduce~\eqref{eq:recursion1} at the parareal iteration $k+1$. This
concludes the proof.
%\hfill $\diamond$
\end{rem}

\subsection{Error bounds for Algorithm~\ref{algorithm-lifting}\label{sec:conv-lifting}}

We consider Algorithm~\ref{algorithm-lifting}, where the
reconstruction at each parareal iteration is done using the lifting
operator $\L$ defined by~\eqref{eq:choix_L}. We show in this section
that the accuracy of the numerical trajectory does not improve 
(neither at the microscale nor at the macroscale) as the
number of parareal iterations $k$ goes to infinity. 

\smallskip

\begin{thm} 
\label{thm:lifting}
Consider Algorithm~\ref{algorithm-lifting}, where $\F_{\Delta t}$ is the
exact propagator of the microscopic problem~\eqref{eq:lin_micro},
$\C_{\Delta t}$ is the exact propagator of the associated macroscopic
problem~\eqref{eq:lin_macro}, and $\L$ is the lifting operator defined
by~\eqref{eq:choix_L}. We fix the time range $[0,T]$, and recall that
the size of the boundary layer $t^{\rm BL}_\epsilon$
in~\eqref{eq:lin_micro} is defined by~\eqref{eq:def_t_BL}.

Then, there exists $\epsilon_0 \in (0,1)$, that only depends on $A$, $q$, $p$,
$\alpha$ 
and $T$, such that, for all $\epsilon<\epsilon_0$ and all $\Delta t >
t^{\rm BL}_\epsilon$, there exists $C$, that depends
on $A$, $q$, $p$, $\alpha$, $\Delta t$ and $T$ such that
\begin{gather}
\label{eq:thm-lifting-macro}
\sup_{0 \leq n \leq N} \left|E^n_0\right| \le C \epsilon
\quad \text{and, for all $k \geq 1$,} \quad
\sup_{0 \leq n \leq N} \left|E^n_k\right| \le C \epsilon^2,
\\  
\label{eq:thm-lifting-micro}
\text{for all $k \geq 0$,} \quad
\sup_{0 \leq n \leq N} \left\| e^n_k \right\| \le C \epsilon,
\end{gather}
where $N=T/\Delta t$ and 
where the macroscopic (resp. microscopic) error $E^n_k$ (resp. $e^n_k$)
is defined by~\eqref{eq:def_e_macro}
(resp.~\eqref{eq:def_e_micro}). Note that $C$ is independent from
$\epsilon$ and $k$.
\end{thm}

\smallskip

The numerical experiments described in Section~\ref{sec:num} show that
these error estimates are sharp. Recall that 
$t^{\rm BL}_\epsilon = C \epsilon \ln(1/\epsilon)$ for some constant $C$
only depending on the matrix $A$ of~\eqref{eq:lin_micro}
(see~\eqref{eq:def_t_BL}). The assumption $\Delta t > 
t^{\rm BL}_\epsilon = C \epsilon \ln(1/\epsilon)$ is therefore
automatically satisfied for sufficiently small $\epsilon$, and in
particular when the time-step $\Delta t$ is of the order of the
macroscopic time scale.

\smallskip

\begin{proof}
Using the definitions~\eqref{eq:def_e_micro} and~\eqref{eq:def_e_macro}
of the microscopic and the macroscopic errors, and the fact that the
microscopic state is reconstructed via the lifting operator~$\L$,  
\[
u_k^n=\L X_k^n,
\]
we have 
$$
\L\, E^n_k 
= 
\L\, X_k^n - \L\R\,\Phi_{\Delta t}^n u_0
= 
u_k^n - \L\R\,\Phi_{\Delta t}^n u_0
=
e_k^n + \Phi_{\Delta t}^n u_0 - \L\R\,\Phi_{\Delta t}^n u_0,
$$
or, equivalently,
\begin{equation}
\label{eq:utile}
e_k^n = \L\, E^n_k - \left(\Id - \L\R\right)\,\Phi_{\Delta t}^n u_0.
\end{equation}
As a consequence, we can write the recursion~\eqref{eq:error-recursion}
for $E_{k+1}^n$ in terms of $E_k^p$ only, by eliminating the microscopic
errors $e_k^p$. We have
\begin{align}
\R\Phi_{\Delta t} e_k^{p}-\rho_{\Delta t}E_k^{p} 
& = \R\Phi_{\Delta t}\left[ \L\, E^p_k - \left(\Id -
    \L\R\right)\,\Phi_{\Delta t}^p u_0\right]-\rho_{\Delta t}E_k^{p}
\nonumber\\
&= \left(\R\Phi_{\Delta t} \L -\rho_{\Delta t}\right)\, E^p_k -
\R\Phi_{\Delta t}\left(\Id - \L\R\right)\,\Phi_{\Delta t}^p u_0.
\label{eq:term}
\end{align}
The first term in~\eqref{eq:term} stems from the difference between the
macroscopic
evolution of the microscopic system and the evolution of the approximate
macroscopic equation. The second term stems from the difference in
evolution between a microscopic state and the (unique) microscopic state
that is obtained by lifting its restriction.
To bound the first term in \eqref{eq:term}, we observe that 
\begin{equation}\label{eq:cts-bound}
\left|\R\Phi_{\Delta t} \L -\rho_{\Delta t}\right|\le C \epsilon.
\end{equation}
Consider indeed the system~\eqref{eq:lin_micro} with the initial
condition $(x_0,y_0) = \L(x_0)$. Then $z_0 = 0$ (because $\L(x_0) \in
\Sigma$, see~\eqref{eq:choix_L}), and we deduce
from~\eqref{eq:x-est-lem} that
$$
|x(\Delta t)-x_0 \exp(\lambda \Delta t)| \le C\epsilon |x_0|,
$$
that reads
$$
\left|\R \Phi_{\Delta t} \L(x_0) - \rho_{\Delta t} x_0 \right| \le 
C \epsilon |x_0|,
$$
from which we infer~\eqref{eq:cts-bound}.

We now turn to the second term of equation~\eqref{eq:term}. We
introduce the shorthand notation for the exact solution
$$
\Phi_{\Delta t}^pu_0 
= \tilde{u}^p = 
\left(\R \tilde{u}^p, \R^\perp \tilde{u}^p\right)=
\left(\tilde{x}^p,\tilde{y}^p\right).
$$
First, using the definition~\eqref{eq:choix_L} for $\L$, we write 
$$
\left(\Id-\L\R\right)\Phi_{\Delta t}^pu_0
= \begin{bmatrix}0\\
\tilde{y}^p - (A^{-1}q)\,\tilde{x}^p
\end{bmatrix}.
$$
Second, using~\eqref{eq:x-est-cor} with the initial condition
$\overline{x}_0=0$, $\overline{y}_0=\tilde{y}^p -
(A^{-1}q)\,\tilde{x}^p$, we get  
\begin{equation}
\left| \R\Phi_{\Delta t}\left(\Id-\L\R\right) \Phi_{\Delta t}^pu_0 \right|
\le C \epsilon \, \left\| \tilde{y}^p- (A^{-1}q) \,\tilde{x}^p \right\|.
\label{eq:term2_pre}
\end{equation}
We are now left with bounding $\left\|
  \tilde{y}^p-(A^{-1}q)\,\tilde{x}^p\right\|$. We note that
$\tilde{y}^p-(A^{-1}q)\,\tilde{x}^p
=z(p \Delta t)$, thus, using~\eqref{eq:z-est-lem} for the solution $u(p
\Delta t) = \Phi_{\Delta t}^pu_0$, we deduce that
$$
\left\| \tilde{y}^p-(A^{-1}q)\,\tilde{x}^p \right\|
=
\left\| z(p \Delta t) \right\|
\leq
C \eps \left( |x_0| + \| z_0 \| \right)
\leq
C \eps. 
$$
Note that we have used the fact that $p \geq 1$ and $\Delta t \geq
t^{\rm BL}_\eps$, hence $p \Delta t \geq t^{\rm BL}_\eps$. We then deduce
from~\eqref{eq:term2_pre} that 
\begin{equation}
\left| \R\Phi_{\Delta t}\left(\Id-\L\R\right) \Phi_{\Delta t}^pu_0 \right|
\le C \epsilon^2.
\label{eq:term2}
\end{equation}
Collecting~\eqref{eq:term},~\eqref{eq:cts-bound} and~\eqref{eq:term2},
we obtain
$$
\left|\R\Phi_{\Delta t} e_k^{p}-\rho_{\Delta t}E_k^{p} \right|
\le C\epsilon\left( \left|E_k^p\right| + \epsilon\right), 
$$
where $C$ only depends on $A$, $q$, $p$, $\alpha$, $T$ and $\Delta t$.

Inserting this bound into the error
recursion~\eqref{eq:error-recursion}, and using that $\rho_{\Delta t} >
0$ (see~\eqref{eq:lin_macro_ex}), we get
$$
\left| E^n_{k+1} \right| \le C \epsilon \sum_{p=1}^{n-1}
\rho^{n-p-1}_{\Delta t}\left(\left|E_k^p\right|+\epsilon\right).
$$
We now introduce $\widetilde{E}_k := \max_{0 \leq n  \leq T/\Delta t}
\left| E_k^n\right|$, and write
$$
\left| E^n_{k+1} \right| 
\le 
C \epsilon 
\left(\widetilde{E}_k +\epsilon\right)
\sum_{p=1}^{n-1} \rho^{n-p-1}_{\Delta t}
=
C \epsilon 
\left(\widetilde{E}_k +\epsilon\right)
\dfrac{1-\rho_{\Delta t}^{n-1}}{1-\rho_{\Delta t}}.
$$
Let $\dps m := \max_{0 \leq n \leq N} 
\dfrac{1-\rho_{\Delta t}^n}{1-\rho_{\Delta t}}$, which only depends on
$\Delta t$, $T$ and $\lambda$. We obtain
$$
\widetilde{E}_{k+1} 
\le 
C m \epsilon \left(\widetilde{E}_k +\epsilon\right),
$$
where $Cm$ only depends on $A$, $q$, $p$, $\alpha$, $T$ and $\Delta
t$ (and is in particular independent of $k$ and $\epsilon$). We thus have
$$
0 \leq \widetilde{E}_k \leq v_k
$$
where the sequence $\left\{ v_k \right\}_{k \in \NN}$ is recursively
defined by $v_{k+1} = C m \epsilon \left(v_k +\epsilon\right)$ and 
$v_0 = \widetilde{E}_0$, so that
$$
v_k=\widetilde{E}_0 (Cm\epsilon)^k + Cm \, \epsilon^2 \, 
\frac{1-(Cm\epsilon)^k}{1-Cm\epsilon}.
$$
Note that the bound~\eqref{eq:x-est-lem} reads
$| x(t) - X(t) | \leq C \epsilon$, hence $v_0 = \widetilde{E}_0 \leq C
\epsilon$. 

Let us choose $\epsilon_0 = 1/(Cm)$. Notice that $\epsilon_0$ only
depends on 
$A$, $q$, $p$, $\alpha$, $T$ and $\Delta t$. For any
$\epsilon \in (0, \epsilon_0)$, the sequence $v_k$ has a limit as $k$ goes to infinity and there exists $C$, independent of $k$ and $\epsilon$, such that
$$
0 \leq \widetilde{E}_0 \leq C \epsilon 
\quad \text{and} \quad
\forall k \geq 1, \ \
0 \leq \widetilde{E}_k \leq v_k \leq C \epsilon^2.
$$
This proves the bound~\eqref{eq:thm-lifting-macro} on the macroscopic error.

To prove the error bound on the microscopic error, we notice, using the
definition~\eqref{eq:choix_L} of $\L$, that 
$$
\left(\Id - \L\R\right)\,\Phi_{\Delta t}^n u_0
=
\left(\Id - \L\R\right)\, u(n \Delta t)
=
\Big( 0,y(n \Delta t) - (A^{-1} q) x(n \Delta t) \Big).
$$
Since $\Delta t \geq t^{\rm BL}_\epsilon$, we deduce
from~\eqref{eq:z-est-lem} that
\begin{equation}
\forall n \geq 1, \quad
\left\| \left(\Id - \L\R\right)\,\Phi_{\Delta t}^n u_0 \right\| 
\leq C \epsilon.
\label{eq:zozo}
\end{equation}
Collecting~\eqref{eq:utile},~\eqref{eq:thm-lifting-macro}
and~\eqref{eq:zozo}, we obtain, for any $k \geq 0$, 
$$
\forall n \geq 1, \quad
\left\| e_k^n \right\| 
\leq C |E^n_k| + \left\| 
\left(\Id - \L\R\right)\,\Phi_{\Delta t}^n u_0 \right\|
\leq
C \epsilon.
$$
Note that the microscopic error is always dominated by the lifting error
(the second term of~\eqref{eq:utile}).

Since, at any parareal iteration $k$, we start with the correct initial
condition, we have $e_k^0 = 0$ and we thus have
proved~\eqref{eq:thm-lifting-micro}. 
\end{proof}

\subsection{Error bounds for Algorithm~\ref{algorithm-matching}}
\label{sec:ana_matching}

We now consider Algorithm~\ref{algorithm-matching}, where the
reconstruction at each parareal iteration is done using {\em any}
matching operator $\P$ satisfying~\eqref{eq:consistance}
and~\eqref{eq:P_fundamental}. The continuity assumption~\eqref{eq:P_b}
will be added when needed. As pointed out above, we do not assume
any specific expression for $\P$ here. We show in
Section~\ref{sec:conv-proj} that, in contrast to
Algorithm~\ref{algorithm-lifting}, the convergence rate obtained with Algorithm~\ref{algorithm-matching}
increases as the number of parareal iterations $k$ increases. Before
that, we show in Section~\ref{sec:exact}  
that, at a given parareal iteration $k$, the computed trajectories
(again both at the macro and the micro scales) are exact up to the time
$k \Delta t$. 

\subsubsection{Local exactness of the algorithm\label{sec:exact}}

One of the important properties of the parareal
algorithm~\eqref{eq:scheme-parareal} is that it 
results, after $k$ parareal iterations, in a solution that is exact at all
times up to $k \Delta t$. The word ``exact'' here means that the parareal
solution is equal to the solution that would have been obtained using only, in
a sequential fashion, the fine-scale propagator up to time $k \Delta
t$. We now show that this exactness property holds for the micro-macro
parareal algorithm we propose.

\smallskip

\begin{thm} 
Consider Algorithm~\ref{algorithm-matching}, where $\F_{\Delta t}$ is the
exact propagator of the microscopic problem~\eqref{eq:lin_micro},
$\C_{\Delta t}$ is the exact propagator of the associated macroscopic
problem~\eqref{eq:lin_macro}, $\L$ is the lifting operator defined
by~\eqref{eq:choix_L} and $\P$ is a matching operator
satisfying~\eqref{eq:consistance} and~\eqref{eq:P_fundamental}. 

Denote by $u_k^n$ the microscopic solution obtained at the $n$-th
time-step and $k$-th parareal iteration, using
Algorithm~\ref{algorithm-matching}. Then, at any parareal iteration $k
\geq 1$, we have
\begin{equation}
\label{eq:thm-exactness}
\forall p \leq k, \quad u_k^p = \Phi_{\Delta t}^p u_0.
\end{equation}
\end{thm}

\begin{proof}
The proof goes by induction. Consider the parareal iteration $k=1$. We
obviously have $u_1^0 = u_0 = \Phi_{\Delta t}^0 u_0$. At time
iteration $n=1$, in view of~\eqref{eq:parareal-projection}, we have
$$
u_1^1 = \P(X_1^1,\overline{u}_0^1),
$$
with (see~\eqref{eq:parareal-propagate})
$$
X_1^1 
= 
\C_{\Delta t}(X_1^0) + \R(\overline{u}_0^1) - \overline{X}_0^1
=
\C_{\Delta t}(X_1^0) + \R(\overline{u}_0^1) - \C_{\Delta t}(X_0^0)
=
\R(\overline{u}_0^1).
$$
Hence, using the fundamental property~\eqref{eq:P_fundamental},
$$
u_1^1 = \P(\R(\overline{u}_0^1),\overline{u}_0^1) = \overline{u}_0^1 = 
\F_{\Delta t}(u_0^0) = \Phi_{\Delta t} u_0.
$$
This proves~\eqref{eq:thm-exactness} for $k=1$. 

Assume now that, at some parareal iteration $k \geq 1$, we
have~\eqref{eq:thm-exactness}. In view of~\eqref{eq:consis}, this
implies that $X_k^p = \R\left[\Phi_{\Delta t}^p u_0\right]$ for any $p
\leq k$. Using~\eqref{eq:start-exact} and the fact that $\Phi_{\Delta t}
u_k^{p-1} = \Phi_{\Delta t}^p u_0$ for all $p \le k$, we deduce that
\begin{equation}
\label{eq:macro_is_exact}
\forall n \leq k+1, \quad
X^n_{k+1}=\R\Phi_{\Delta t} u_k^{n-1} = \R\Phi_{\Delta t}^n u_0.
\end{equation}
Hence, at the parareal iteration $k+1$, the macroscopic solution is
exact up to time $(k+1) \Delta t$. Using~\eqref{eq:parareal-projection},
we now write, for any $n \leq k$,
$$
u_{k+1}^{n+1}
=
\P(X^{n+1}_{k+1},\overline{u}_k^{n+1})
=
\P(\R\Phi_{\Delta t}^{n+1} u_0,\Phi_{\Delta t} u_k^n)
=
\P(\R\Phi_{\Delta t}^{n+1} u_0,\Phi_{\Delta t}^{n+1}u_0)
=\Phi_{\Delta t}^{n+1}u_0,
$$
where we have used~\eqref{eq:macro_is_exact} and~\eqref{eq:micro_prop}
in the first equality, the exactness assumption of the microscopic
solution at iteration $k$ in the second equality, and the fundamental
property~\eqref{eq:P_fundamental} of the matching operator $\P$ in the
last equality. This proves the relation~\eqref{eq:thm-exactness} at the
iteration $k+1$ and concludes the proof. 
\end{proof}

This result also directly follows from our above remark that, in its form
\eqref{eq:para_scheme_match_expl}, Algorithm 2 is of the form \eqref{eq:scheme-parareal}.

\subsubsection{Error bounds\label{sec:conv-proj}}

We now establish error bounds on Algorithm~\ref{algorithm-matching} that
show that the microscopic solution converges towards the exact
microscopic dynamics when the modeling error $\epsilon$ decreases, and
that the convergence rate improves as the number of parareal iterations
$k$ increases. This is in contrast with 
Algorithm~\ref{algorithm-lifting}, where the error does not improve even
if $k$ goes to infinity (see Section~\ref{sec:conv-lifting}). With
Algorithm~\ref{algorithm-matching}, we recover the behavior of the
standard parareal algorithm, as recalled in Remark~\ref{rem:std_ananu}
(see e.g.~\eqref{eq:recursion1}). 

\smallskip

\begin{thm}\label{thm:matching} 
Consider Algorithm~\ref{algorithm-matching}, where $\F_{\Delta t}$ is the
exact propagator of the microscopic problem~\eqref{eq:lin_micro},
$\C_{\Delta t}$ is the exact propagator of the associated macroscopic
problem~\eqref{eq:lin_macro}, $\L$ is the lifting operator defined
by~\eqref{eq:choix_L}, and $\P$ is a matching operator
satisfying~\eqref{eq:consistance},~\eqref{eq:P_fundamental}
and~\eqref{eq:P_b}. We fix the time range $[0,T]$, and recall that
the size of the boundary layer $t^{\rm BL}_\epsilon$
in~\eqref{eq:lin_micro} is defined by~\eqref{eq:def_t_BL}.

Then, there exists $\epsilon_0 \in (0,1)$, that only depends on $A$,
$q$, $p$, $\alpha$ and $T$, such that, for all $\epsilon<\epsilon_0$ and
all $\Delta t > t^{\rm BL}_\epsilon$, there exists a constant $C_k$,
independent of $\epsilon$, such that
\begin{gather}
\label{eq:thm-matching-macro}
\text{for all $k \geq 0$,} \quad
\sup_{0 \leq n \leq N} \left|E^n_k\right| \le C_k 
\epsilon^{1+\lceil k/2\rceil},
\\  
\label{eq:thm-matching-micro}
\text{for all $k \geq 0$,} \quad
\sup_{0 \leq n \leq N} \left\| e^n_k \right\| \le C_k 
\epsilon^{1+\lfloor k/2\rfloor},
\end{gather}
where $N = T/ \Delta t$ and 
where the macroscopic (resp. microscopic) error $E^n_k$ (resp. $e^n_k$)
is defined by~\eqref{eq:def_e_macro}
(resp.~\eqref{eq:def_e_micro}).
The constant $C_k$ is independent from $\epsilon$, but a priori depends on
$k$, $A$, $q$, $p$, $\alpha$, $\Delta t$ and $T$.
\end{thm}

\smallskip

In~\eqref{eq:thm-matching-macro} and~\eqref{eq:thm-matching-micro}, we
used the notation: for any $x \in \RR$, $\lceil x \rceil \in \Z$ and
$\lfloor x \rfloor \in \Z$ are respectively defined by: $\lfloor x
\rfloor \le x < \lfloor x \rfloor +1$ and $\lceil x \rceil - 1 < x \le
\lceil x \rceil$. 

The above result shows that the parareal iterations alternatingly
improve the macroscopic and the microscopic errors by an order of
magnitude in $\epsilon$. The numerical results of Section~\ref{sec:num}
show that~\eqref{eq:thm-matching-macro}
and~\eqref{eq:thm-matching-micro} are sharp error estimates. As already
mentioned above, the assumption $\Delta t > t^{\rm BL}_\epsilon$ is
automatically satisfied for sufficiently small 
$\epsilon$, in particular when the time-step $\Delta t$ is of the
order of the macroscopic time-scale.

The bounds~\eqref{eq:thm-matching-macro}
and~\eqref{eq:thm-matching-micro} show that, as $k$ increases, the rate
of convergence (with respect to $\epsilon$) of the error increases. The
dependence of the constant $C_k$ in these two bounds on $\Delta t$ and $k$
will be analyzed in details on the numerical test case considered in
Section~\ref{sec:num-exact} (see~\eqref{eq:obs3_macro}
and~\eqref{eq:obs3_micro}). 

\smallskip

\begin{proof}
Using~\eqref{eq:parareal-projection},~\eqref{eq:P_fundamental} and the
definition~\eqref{eq:def_e_micro} of the microscopic error, we have
$$
e_{k+1}^n
=
u^n_{k+1} - \Phi_{\Delta t}^nu_0
=
\P(X^n_{k+1},\overline{u}_{k}^n) - 
\P(\R(\Phi_{\Delta t}^n u_0),\Phi_{\Delta t}^n u_0).
$$
Hence, using~\eqref{eq:P_b}, we deduce that
\begin{align}
\left\| e_{k+1}^n \right\| &\le C \left(
\left\| \overline{u}_{k}^n - \Phi_{\Delta t}^n u_0 \right\|
+
\left| X^n_{k+1} - \R(\Phi_{\Delta t}^n u_0) \right|
\right)
\nonumber \\
& \le C \left(
\left\| \Phi_{\Delta t} u_{k}^{n-1} - \Phi_{\Delta t}^n u_0 \right\|
+
\left| E_{k+1}^n \right|
\right)
\nonumber \\
& \le C \left(
\left\| \Phi_{\Delta t} e_{k}^{n-1} \right\|
+
\left| E_{k+1}^n \right| 
\right). 
\label{eq:err-micro_pre}
\end{align}
Since $\Delta t \geq t^{\rm BL}_\epsilon$, we infer from~\eqref{eq:x-est-cor}
and~\eqref{eq:y-est-cor} that 
$$
\left\| \Phi_{\Delta t} e_{k}^{n-1} \right\| 
\leq
C \left( \left| \R e_{k}^{n-1} \right| + \epsilon \left\| \R^\perp
    e_{k}^{n-1} \right\| \right)
=
C \left( \left| E_{k}^{n-1} \right| + \epsilon \left\| \R^\perp
    e_{k}^{n-1} \right\| \right),
$$
where we have used~\eqref{eq:consis_e}. We then deduce
from~\eqref{eq:err-micro_pre} that
\begin{equation}
\left\| e_{k+1}^n \right\| 
\le 
C \left(
\left| E_{k}^{n-1} \right| + \epsilon \left\| \R^\perp
    e_{k}^{n-1} \right\| 
+
\left| E_{k+1}^n \right| 
\right)
\le 
C \left(\left| E_{k}^{n-1} \right| + 
\epsilon \left\| e_{k}^{n-1} \right\| +
\left| E_{k+1}^{n} \right| \right).
\label{eq:err-micro}
\end{equation}

We now bound the macroscopic error $E_{k+1}^n$, using the recursion
formula~\eqref{eq:error-recursion}, that reads
\begin{equation}
E^n_{k+1} 
= 
\sum_{p=1}^{n-1} \rho_{\Delta t}^{n-p-1} \, T^p_k,
\label{eq:error-recursion_bis}
\end{equation}
with $T^p_k := \R\Phi_{\Delta t} e_k^{p}-\rho_{\Delta
  t}E_k^{p}$. Consider the solution $(\widetilde{x}(t),\widetilde{y}(t))$
to the system~\eqref{eq:lin_micro} with initial condition 
$\widetilde{u}(0) = e_k^{p}$, that is $\widetilde{x}(0) = \R(e_k^{p}) =
E_k^p$ and $\widetilde{y}(0) = \R^\perp(e_k^{p})$. We then have 
$$
T^p_k = \widetilde{x}(\Delta t) - \widetilde{X}(\Delta t),
$$
where $\widetilde{X}(\Delta t)$ is the solution to~\eqref{eq:lin_macro}
with initial condition $\widetilde{X}(0) = E_k^p$. In view
of~\eqref{eq:x-est-lem}, we have
$$
\left| T^p_k \right| 
= 
\left| \widetilde{x}(\Delta t) - \widetilde{X}(\Delta t) \right|
\leq
C \eps \left( |\widetilde{x}(0)| + 
\| \widetilde{y}(0) - A^{-1} q \widetilde{x}(0) \| \right)
\leq 
C \eps 
\left(\left|E^p_k\right| + \left\|e^{p}_{k}\right\|\right),
$$
where $C$ is independent from $p$, $k$ and $\epsilon$.
We are now in position to use the
recursion~\eqref{eq:error-recursion_bis}, from which we infer
\begin{equation}
\left| E^n_{k+1} \right| 
\le 
\sum_{p=1}^{n-1}\rho^{n-p-1}_{\Delta t}\left|T^p_k\right|
\le 
C \epsilon \sum_{p=1}^{n-1} \rho^{n-p-1}_{\Delta t} 
\left(\left|E_k^p\right|+\left\|e^{p}_{k}\right\|\right). 
\label{eq:err-macro}
\end{equation}
We now prove the theorem by induction, using the two fundamental
estimates~\eqref{eq:err-micro} and~\eqref{eq:err-macro}. At the parareal
iteration $k=0$, Algorithm~\ref{algorithm-matching} is identical to
Algorithm~\ref{algorithm-lifting}. In view
of~\eqref{eq:thm-lifting-macro} and~\eqref{eq:thm-lifting-micro}, we
thus have
$$
\sup_{0 \leq n \leq N} \left| E_0^n \right| \le C_0 \epsilon
\quad \text{and} \quad 
\sup_{0 \leq n \leq N} \left\| e_0^n \right\| \le C_0 \epsilon,
$$
that is~\eqref{eq:thm-matching-macro} and~\eqref{eq:thm-matching-micro}
for $k=0$.

Let us now assume that~\eqref{eq:thm-matching-macro}
and~\eqref{eq:thm-matching-micro} hold at any iteration $k' \leq k$,
with $k$ an even integer. We 
prove the bounds at iteration $k+1$. Setting $m=k/2+1$ (so that $\lfloor
k/2 \rfloor=\lceil k/2 \rceil=m-1$, $\lfloor (k-1)/2 \rfloor=m-2$ and
$\lceil (k-1)/2 \rceil=m-1$), we thus assume that
\begin{gather*}
\sup_{0 \leq p \leq N} \left|E_{k-1}^p\right| \le C_{k-1} \epsilon^m,
\quad
\sup_{0 \leq p \leq N} \left\| e_{k-1}^p \right\| \le C_{k-1} \epsilon^{m-1},
\\
\sup_{0 \leq p \leq N} \left|E_k^p\right| \le C_k \epsilon^m,
\quad
\sup_{0 \leq p \leq N} \left\| e_k^p \right\| \le C_k \epsilon^m.
\end{gather*}
Then, we infer from~\eqref{eq:err-macro} that, for any $0 \leq n \leq N$,
$$
\left| E^n_{k+1} \right| 
\le 
C C_k \epsilon^{m+1} \sum_{p=1}^{n-1} \rho^{n-p-1}_{\Delta t} 
\le 
C C_k \epsilon^{m+1} \dfrac{1-\rho_{\Delta t}^{n-1}}{1-\rho_{\Delta t}}
\le 
\dfrac{C C_k \epsilon^{m+1}}{1-\rho_{\Delta t}} 
\le 
C_{k+1} \epsilon^{m+1},
$$
where $C_{k+1}$ is independent from $\epsilon$, but depends on $k$ and
$\Delta t$. We
next deduce from~\eqref{eq:err-micro} that, for any $0 \leq n \leq N$,
$$
\left\| e_{k+1}^n \right\| 
\le 
C \left( C_k \epsilon^m + C_k \epsilon^{m+1} + C_{k+1} \epsilon^{m+1} \right) 
\le 
\tilde{C}_{k+1} \epsilon^m,
$$
where $\tilde{C}_{k+1}$ is again independent from $\epsilon$, but
depends on $k$. We thus have proved~\eqref{eq:thm-matching-macro}
and~\eqref{eq:thm-matching-micro} at iteration $k+1$. 

We now assume that~\eqref{eq:thm-matching-macro}
and~\eqref{eq:thm-matching-micro} hold at any iteration $k' \leq k$,
with $k$ an odd integer. We 
prove the bounds at iteration $k+1$. Setting $m=(k-1)/2+1$ (so that
$\lfloor k/2 \rfloor=m-1$, $\lceil k/2 \rceil=m$ and $\lfloor (k-1)/2
\rfloor=\lceil (k-1)/2 \rceil=m-1$), we thus assume that
\begin{gather*}
\sup_{0 \leq p \leq N} \left|E_{k-1}^p \right| \le C_{k-1} \epsilon^m,
\quad 
\sup_{0 \leq p \leq N} \left\|e_{k-1}^p \right\| \le C_{k-1} \epsilon^m,
\\  
\sup_{0 \leq p \leq N} \left| E_k^p \right| \le C_k \epsilon^{m+1},
\quad 
\sup_{0 \leq p \leq N} \left| e_k^p \right|\le C_k \epsilon^m.
\end{gather*}
Using again equations~\eqref{eq:err-macro} and~\eqref{eq:err-micro}, we
find that, for any $0 \leq n \leq N$,
$$
\left| E_{k+1}^n \right| \le C_{k+1} \epsilon^{m+1}
\quad \text{and} \quad 
\left\| e_{k+1}^n \right\| \le C_{k+1} \epsilon^{m+1},
$$
where $C_{k+1}$ is again independent from $\epsilon$. We thus have
proved~\eqref{eq:thm-matching-macro} and~\eqref{eq:thm-matching-micro}
at iteration $k+1$. This concludes the proof. 
\end{proof}

\subsection{Error bounds for Algorithm~3\label{sec:anal3}}

Since Algorithm~3 uses the standard parareal
iteration~\eqref{eq:scheme-parareal}, local exactness is automatically
satisfied. We proceed to proving error bounds on Algorithm~3, which can
be compared to those of Algorithm~\ref{algorithm-matching}. 

\smallskip

\begin{thm}
\label{thm:maday} 
Consider Algorithm~3 given by~\eqref{eq:para_scheme_maday}, where
$\F_{\Delta t}$ is the exact propagator of the microscopic
problem~\eqref{eq:lin_micro}, $\G_{\Delta t}$ is the exact propagator of
the associated macroscopic DAE problem~\eqref{eq:coarse_maday}, and $\L$
is the lifting operator defined by~\eqref{eq:choix_L}. We fix the time
range $[0,T]$, and recall that the size of the boundary layer $t^{\rm
  BL}_\epsilon$ in~\eqref{eq:lin_micro} is defined by~\eqref{eq:def_t_BL}.

Then, there exists $\epsilon_0 \in (0,1)$, that only depends on $A$, $q$, $p$,
$\alpha$ and $T$, such that, for all $\epsilon<\epsilon_0$ and all
$\Delta t > t^{\rm BL}_\epsilon$, there exists a constant $C_k$,
independent of $\epsilon$, such that 
\begin{gather}
\label{eq:thm-maday-macro}
\text{for all $k \geq 0$,} \quad
\sup_{0 \leq n \leq N} \left|E^n_k\right| \le C_k 
\epsilon^{k+1},
\\  
\label{eq:thm-maday-micro}
\text{for all $k \geq 0$,} \quad
\sup_{0 \leq n \leq N} \left\| e^n_k \right\| \le C_k 
\epsilon^{k+1},
\end{gather}
where $N = T/ \Delta t$ and 
where the macroscopic (resp.~microscopic) error $E^n_k$ (resp.~$e^n_k$)
is defined by~\eqref{eq:def_e_macro}
(resp.~\eqref{eq:def_e_micro}). The constant $C_k$ is independent from
$\epsilon$, but a priori depends on $k$, $A$, $q$, $p$, $\alpha$,
$\Delta t$ and $T$.
\end{thm}

\smallskip

These results are in agreement with~\cite[Theorem
2.1]{maday41parareal}.
The numerical results of Section~\ref{sec:num}
show that~\eqref{eq:thm-maday-macro}
and~\eqref{eq:thm-maday-micro} are sharp error estimates. 

\smallskip

The above result shows that, in contrast to
Algorithm~\ref{algorithm-matching}, Algorithm~3 improves the order of
convergence (in terms of $\epsilon$) of both the
macroscopic and the microscopic errors by an order of magnitude in
$\epsilon$ at each iteration. As noted in
Section~\ref{sec:compare_scheme}, this improved convergence rate comes
at the price of a larger computational cost per iteration. 

\smallskip

\begin{proof}
Using~\eqref{eq:para_scheme_maday} and the
definition~\eqref{eq:def_e_micro} of the microscopic error, we have
$$
e_{k+1}^n
=
u^n_{k+1} - \Phi_{\Delta t}^nu_0
=
\Phi_{\Delta t}e^{n-1}_k+\L \rho_{\Delta t}\R\left(e^{n-1}_{k+1}-e^{n-1}_k\right).
$$
Using~\eqref{eq:consis_e}, we deduce from the above equation that
\begin{align}
\left\| e_{k+1}^n \right\| 
& \le 
\left\| 
\Phi_{\Delta t} e_{k}^{n-1} -\L \rho_{\Delta t}\R e^{n-1}_k
\right\|
+
\left\| \L \rho_{\Delta t} E^{n-1}_{k+1} \right\|. 
\label{eq:err-micro_maday-2}
\end{align}
The first term is decomposed as 
\begin{multline*}
\left\| 
\Phi_{\Delta t} e_{k}^{n-1} -\L \rho_{\Delta t}\R e^{n-1}_k
\right\| 
\\
\le 
	C \left(
	\left| 
	\R\Phi_{\Delta t} e_{k}^{n-1} -\rho_{\Delta t}\R e^{n-1}_k
	\right|
	+
	\left\| 
	\R^\perp\Phi_{\Delta t} e_{k}^{n-1} -R^{\perp}\L\rho_{\Delta t}\R e^{n-1}_k
	\right\|
	\right).
\end{multline*}
Since $\Delta t \geq t^{\rm BL}_\epsilon$, we infer an estimate on the
first (resp.~second) term of the above right-hand side
using~\eqref{eq:x-est-lem} (resp.~\eqref{eq:z-est-lem}), resulting in
\begin{equation}
\left\| 
\Phi_{\Delta t} e_{k}^{n-1} -\L \rho_{\Delta t}\R e^{n-1}_k
\right\| 
\le 
C \eps \left\| e_k^{n-1} \right\|.
\label{eq:tutu}
\end{equation}
Collecting~\eqref{eq:err-micro_maday-2} and~\eqref{eq:tutu}, we deduce
that there exists a constant $C$, independent of $\epsilon$, such that
\begin{equation}
\left\| e_{k+1}^n \right\|
\le C  
\left(
\eps\left\| e_{k}^{n-1}\right\|+ \left|E^{n-1}_{k+1}\right| 
\right).
\label{eq:err-micro_maday-3}
\end{equation}
This estimate is to be compared with~\eqref{eq:err-micro} in the proof of
Theorem~\ref{thm:matching}. 
Using~\eqref{eq:err-micro_maday-3}, the proof of Theorem~\ref{thm:maday}
is completed via induction, in a way that is similar to (but simpler
than) the proof of Theorem~\ref{thm:matching}. 
\end{proof}

\section{Numerical experiments (linear test-case)\label{sec:num}}

In this section, we numerically illustrate the above convergence
results on a linear problem. We first consider the case when both the
microscopic and the macroscopic models are integrated exactly
(Section~\ref{sec:num-exact}). We next consider the case when the
macroscopic propagator is a forward Euler discretization, thus
introducing some finite step-size error
(Section~\ref{sec:num-macro-fe}). 

We consider the example system
\begin{equation}
\label{eq:toy-problem}
\begin{cases}
\phantom{al}\dot{x} &= \phantom{al} -\dfrac{x}{2} - \dfrac{y_1 + y_2}{4}
\\ \noalign{\vskip 5pt}
\begin{bmatrix}\dot{y}_1 \\ \dot{y}_2\end{bmatrix} 
&=\phantom{al} \dfrac{1}{\eps}
\left(\begin{bmatrix}1 \\ 1\end{bmatrix}x - \begin{bmatrix}1/2 & 1/2 \\ 0 & 1/3\end{bmatrix}\begin{bmatrix}y_1\\y_2\end{bmatrix}\right) 
\end{cases},
\end{equation}
which is of the form~\eqref{eq:lin_micro}. The associated macroscopic,
slow dynamics is given by~\eqref{eq:lin_macro} with $\lambda=-1$.
The initial condition is $x(0)=1$, $y_1(0)=y_2(0)=0$, and we consider the
solution on the interval $[0,T]$ with $T=N\Delta t=10$. 

The fine-scale propagator $\F_{\Delta t}$ is the exact integrator
of~\eqref{eq:toy-problem}. The coarse propagator $\C_{\Delta t}$ is the
exact integrator of~\eqref{eq:lin_macro} in Section~\ref{sec:num-exact},
and a forward Euler discretization of~\eqref{eq:lin_macro} in
Section~\ref{sec:num-macro-fe}.
We choose the parareal time-step $\Delta t= 10^{-1}$, and consider
$\epsilon \in [10^{-5},10^{-1}]$. The lifting operator $\L$ is defined
by~\eqref{eq:choix_L}, and we use the matching operator
$\P$ defined by~\eqref{eq:projection-good}.

We look at the relative macroscopic error $\left|E_k^N\right|/\left|
  x(T)\right|$ and the relative microscopic error
$\left\|e_k^N\right\|/\left\|u(T)\right\|$ at the final time $T = t_N =
N \Delta t$ for
different iteration numbers $k$, satisfying $0\le k \le K$.
%% CHANGED: nothing, based on comment Tony on this already being a lower bound, and the proof being an upper bound, implying sharpness. 
% \fred{revoir cette phrase: il me semblerait mieux de mettre le max de
%   l'erreur plutot 
%   que l'erreur a la fin, ca illustre mieux notre resultat; sauf a
%   montrer que, pour le cas particulier ici envisage, c'est la meme
%   chose, mais je n'en suis pas sur}

\subsection{Results with exact microscopic and macroscopic integrations\label{sec:num-exact}}

In this section, we take both the fine-scale and the coarse propagators
to be the exact integrators.

\subsubsection{Algorithm~\ref{algorithm-lifting}}

We first consider Algorithm~\ref{algorithm-lifting} (analyzed in
Section~\ref{sec:conv-lifting}), where the reconstruction at the end of
each parareal iteration is done using the lifting operator $\L$.

We set the maximal number of parareal iterations at
$K=2$. Figure~\ref{fig:num-exact-lifting} shows the macroscopic and
microscopic errors as a function of $\epsilon$ for the chosen values of
$k$. We see that the macroscopic error is of the order of
$O(\epsilon^2)$ as soon as $k\ge 1$ (and is of the order of $O(\epsilon)$ at
$k=0$). The macroscopic error at $k=2$ is equal to that at $k=1$. We also
observe that the microscopic error is always of the order of $\epsilon$
(for any $k$), although the value of the error is smaller at $k=1$ than
at $k=0$. These results are in agreement with Theorem~\ref{thm:lifting},
and confirm the fact that the accuracy of
Algorithm~\ref{algorithm-lifting} does not improve when $k$ goes to
infinity. 

\begin{figure}[htbp]
\includegraphics[scale=0.76]{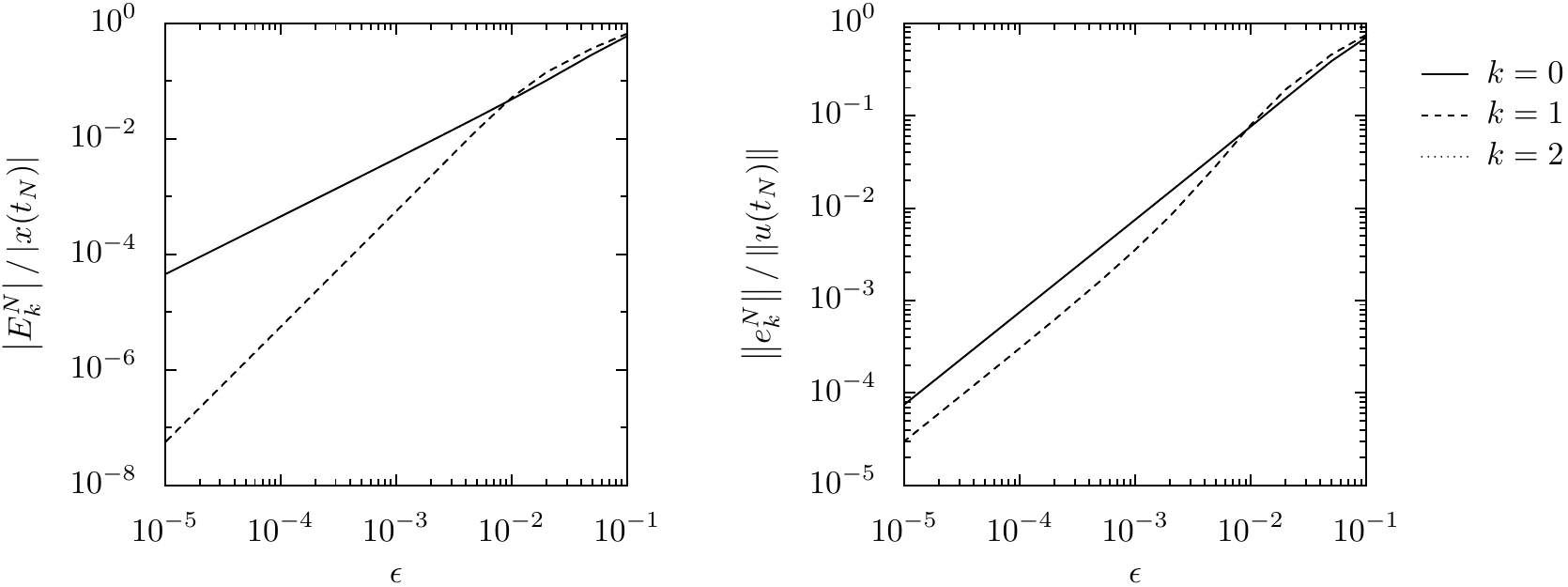}
\caption{\label{fig:num-exact-lifting} 
Algorithm~\ref{algorithm-lifting} for the system~\eqref{eq:toy-problem},
with exact fine-scale and coarse propagators: errors as a function of
$\epsilon$ for different values of $k$ (left: macroscopic error; right:
microscopic error). Note that the lines for $k \ge 1$ visually overlap.}
\end{figure}
  
When $\epsilon$ is too large, the macroscopic error is not anymore of
the order of $O(\epsilon^2)$ at the iteration $k \geq 1$. This is due to the
fact that the assumption 
$\Delta t \geq t^{\rm BL}_\epsilon$ is no longer satisfied. Recall
indeed that we keep
$\Delta t$ fixed, and $t^{\rm BL}_\epsilon = C \eps \ln(1/\eps)$
increases if $\epsilon$ increases. Hence, for too large values of
$\epsilon$, the time step $\Delta t$ is too small to correct for the
initial boundary layer.

\subsubsection{Algorithm~\ref{algorithm-matching}}
\label{sec:num_algo2}

We now consider Algorithm~\ref{algorithm-matching} (analyzed in
Section~\ref{sec:ana_matching}), where the reconstruction at the end of
each parareal iteration is performed using the matching operator $\P$.

The maximal number of parareal iterations is set at
$K=6$. Figure~\ref{fig:num-exact-matching} shows the macroscopic and
microscopic errors as a function of $\epsilon$, for the chosen values of
$k$. The numerical results are in agreement with
Theorem~\ref{thm:matching}. At each odd parareal iteration, the order of
convergence (in terms of $\epsilon$) of the macroscopic error increases
by $1$, whereas the microscopic error decreases, but remains of the same
order in $\epsilon$. At each even iteration, the converse holds: the
order of convergence (in terms of $\epsilon$) increases by $1$ for the
microscopic error, whereas the macroscopic error decreases but its order
remains alike. Note also that, for the smallest considered values of
$\epsilon$, the algorithm reaches machine precision in $5$ to $6$
iterations. 

\begin{figure}[htbp]
\includegraphics[scale=0.76]{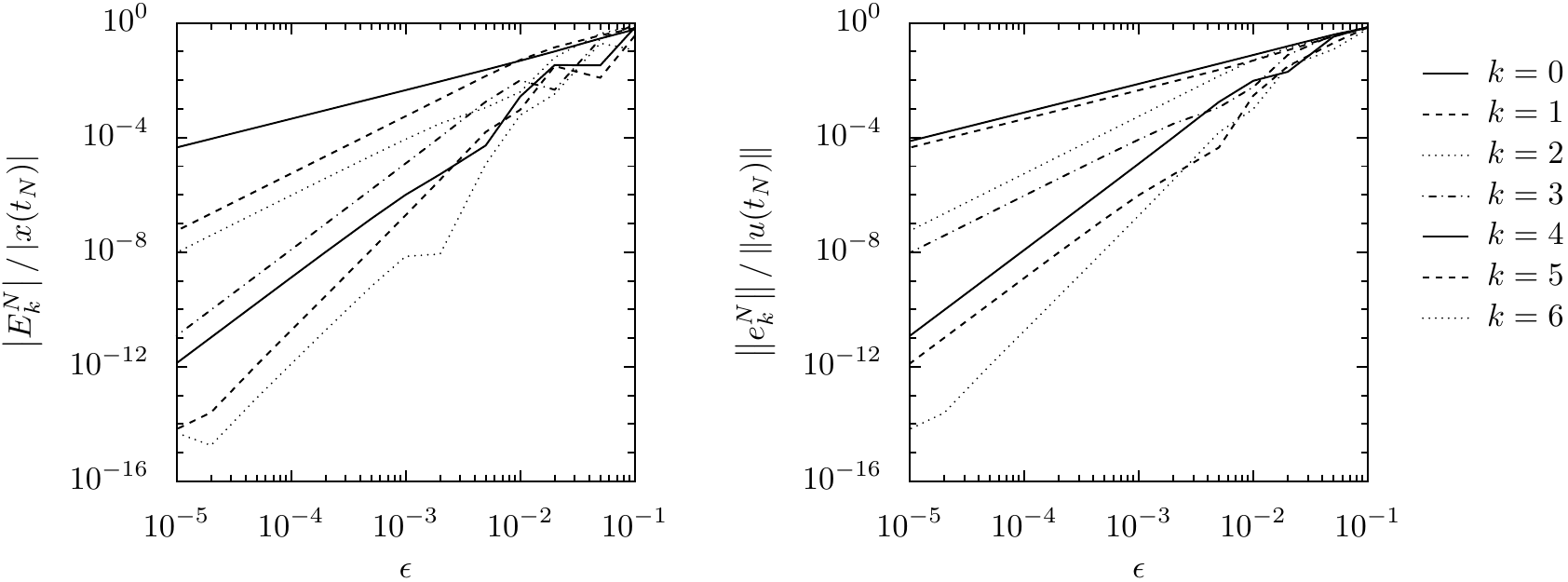}
\caption{\label{fig:num-exact-matching}
Algorithm~\ref{algorithm-matching} for the system~\eqref{eq:toy-problem},
with exact fine-scale and coarse propagators: errors as a function of
$\epsilon$ for different values of $k$ (left: macroscopic error; right:
microscopic error).}
\end{figure}

As with Algorithm~\ref{algorithm-lifting}, when $\epsilon$ is too large, the
numerical results do not agree with the theoretical results, because the
chosen time-step $\Delta t$ does not satisfy the assumption $\Delta t
\geq t^{\rm BL}_\epsilon$. 

\medskip

At this point, we have numerically verified our theoretical results, and
know that the macroscopic error is bounded from above by, and actually
roughly of the order of,
\begin{equation}
\label{eq:obs1} 
\sup_n | E_k^n | \approx 
C_{k,\Delta t} \left( \frac{\eps}{\Delta t} \right)^{1+\lceil
  k/2\rceil},
\end{equation} 
where $C_{k,\Delta t}$ a priori depends on $k$ and $\Delta t$, but is
independent of $\epsilon$ (and likewise for the microscopic error). 

On Figure~\ref{fig:num-exact-matching-k}, we plot the macroscopic and
microscopic errors as a function of the iteration number $k$, $0 \le k
\le K=30$, for various values of $\epsilon$. We observe an exponential
convergence to the exact solution as a function of $k$, with a
convergence rate that increases when $\epsilon$ decreases. We deduce
from~\eqref{eq:obs1} how the constant $C_{k,\Delta t}$ depends on $k$:
there exists $C^0_{\Delta t}$ and $C^1_{\Delta t}$, independent of $k$
and $\epsilon$, such that
\begin{equation}
\label{eq:obs2} 
\sup_n | E_k^n | \approx 
C^0_{\Delta t} \left( C^1_{\Delta t} \right)^{1+\lceil k/2\rceil}
\left( \frac{\eps}{\Delta t} \right)^{1+\lceil k/2\rceil}.
\end{equation} 
Note that, for $\epsilon=10^{-1}$ (which is quite
a large value compared to $\Delta t = 10^{-1}$), the convergence
is very slow, which is in agreement with the previous observations.

\begin{figure}[htbp]
\includegraphics[scale=0.74]{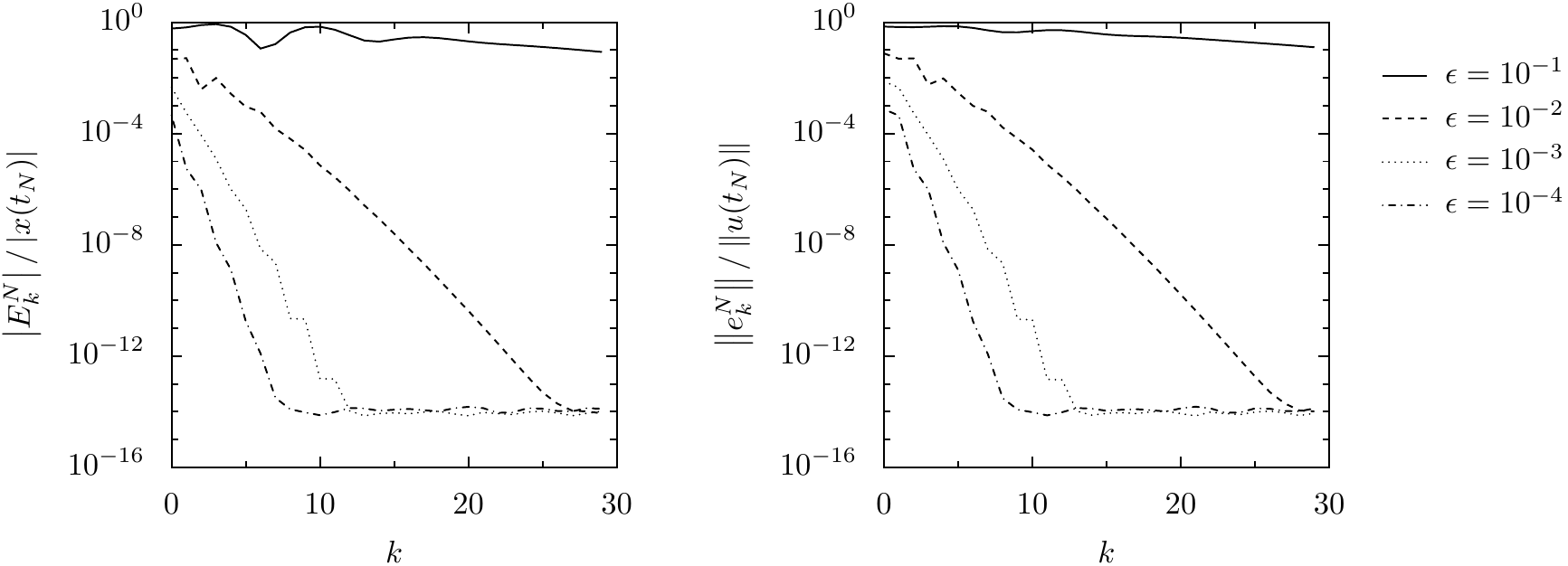}
\caption{\label{fig:num-exact-matching-k}
Algorithm~\ref{algorithm-matching} for the system~\eqref{eq:toy-problem},
with exact fine-scale and coarse propagators and parareal time step
$\Delta t=10^{-1}$: errors as a function of $k$ for different values of
$\epsilon$ (left: macroscopic error; right: microscopic error).}
\end{figure}

To understand how $C^0_{\Delta t}$ and $C^1_{\Delta t}$ depend on
$\Delta t$, we perform another experiment, in which we fix
$\epsilon=10^{-5}$ and vary $\Delta t$.  We then plot the error as a
function of $\Delta t^{-1}$ for different values of $k$ (see
Figure~\ref{fig:num-exact-matching-Dt}). These results show that the 
macroscopic error varies proportionally to $\left( \Delta t^{-1}
\right)^{1+\lceil k/2\rceil}$. Combined with~\eqref{eq:obs2}, we
therefore deduce that, on this numerical test-case, the macroscopic
error satisfies
\begin{equation}
\label{eq:obs3_macro} 
\sup_n | E_k^n | \approx 
C^0_{\rm macro} 
\left( C^1_{\rm macro} \frac{\eps}{\Delta t} \right)^{1+\lceil k/2\rceil},
\end{equation} 
and likewise for the microscopic error:
\begin{equation}
\label{eq:obs3_micro} 
\sup_n | e_k^n | \approx 
C^0_{\rm micro} 
\left( C^1_{\rm micro} \frac{\eps}{\Delta t} \right)^{1+\lfloor k/2\rceil}.
\end{equation}
We in particular see that, if $\eps/\Delta t$ is sufficiently small, then the
parareal trajectory converges to the exact trajectory when $k$ goes to
$\infty$. 

\begin{figure}[htbp]
\includegraphics[scale=0.74]{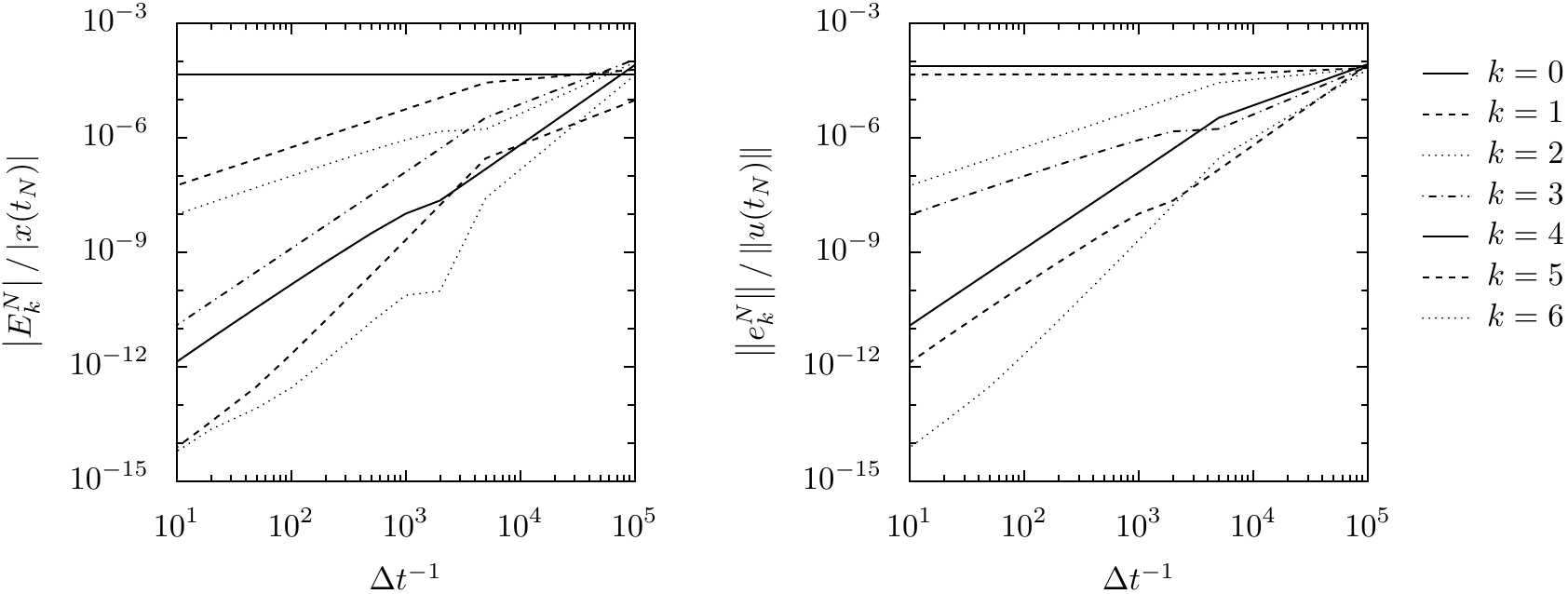}
\caption{\label{fig:num-exact-matching-Dt}
Algorithm~\ref{algorithm-matching} for the system~\eqref{eq:toy-problem},
with exact fine-scale and coarse propagators and  $\epsilon=10^{-5}$: errors as a function of $\Delta t^{-1}$
for different values of $k$ (left: macroscopic error; right:
microscopic error).}
\end{figure}

\subsubsection{Algorithm 3}

To complete these numerical tests, we consider Algorithm 3 originally
proposed in~\cite{BBK,maday41parareal} (which we analyzed in
Section~\ref{sec:anal3}), and repeat the previous experiment. The
results, shown in Figure~\ref{fig:num-exact-maday}, are in agreement
with the theoretical results.

\begin{figure}[htbp]
	\includegraphics[scale=0.76]{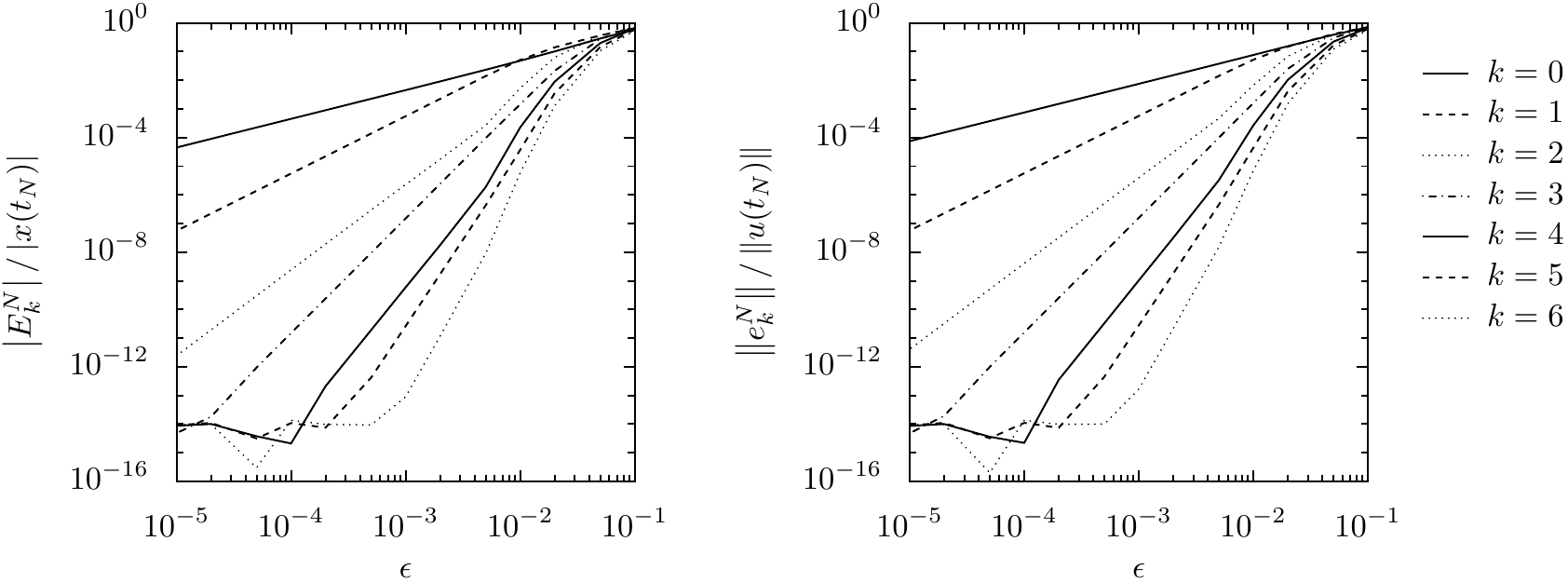}
	\caption{\label{fig:num-exact-maday}
Algorithm 3 for the system~\eqref{eq:toy-problem},
with exact fine-scale and coarse propagators: errors as
a function of $\epsilon$ for different values of $k$ (left: macroscopic
error; right: microscopic error).}
\end{figure}

\subsection{Results with exact microscopic and approximate macroscopic
  integration\label{sec:num-macro-fe}} 

In this section, we repeat the above experiments, but now using a
forward Euler time discretization for the coarse propagator, using the
time step $\Delta t$ (hence, to propagate the system over the time range
$\Delta t$, the scheme $\C_{\Delta t}$ consists in doing a {\em single}
step of the forward Euler  algorithm). The fine-scale propagator is again the
exact one. 

\subsubsection{Algorithm~\ref{algorithm-lifting}}

We first consider Algorithm~\ref{algorithm-lifting} (for which the
reconstruction is done using the lifting operator $\L$), and set the
maximal number of parareal iterations to
$K=3$. Figure~\ref{fig:num-fe-lifting} shows the errors as a function of
$\epsilon$ for the chosen values of $k$.

\begin{figure}[htbp]
	\includegraphics[scale=0.76]{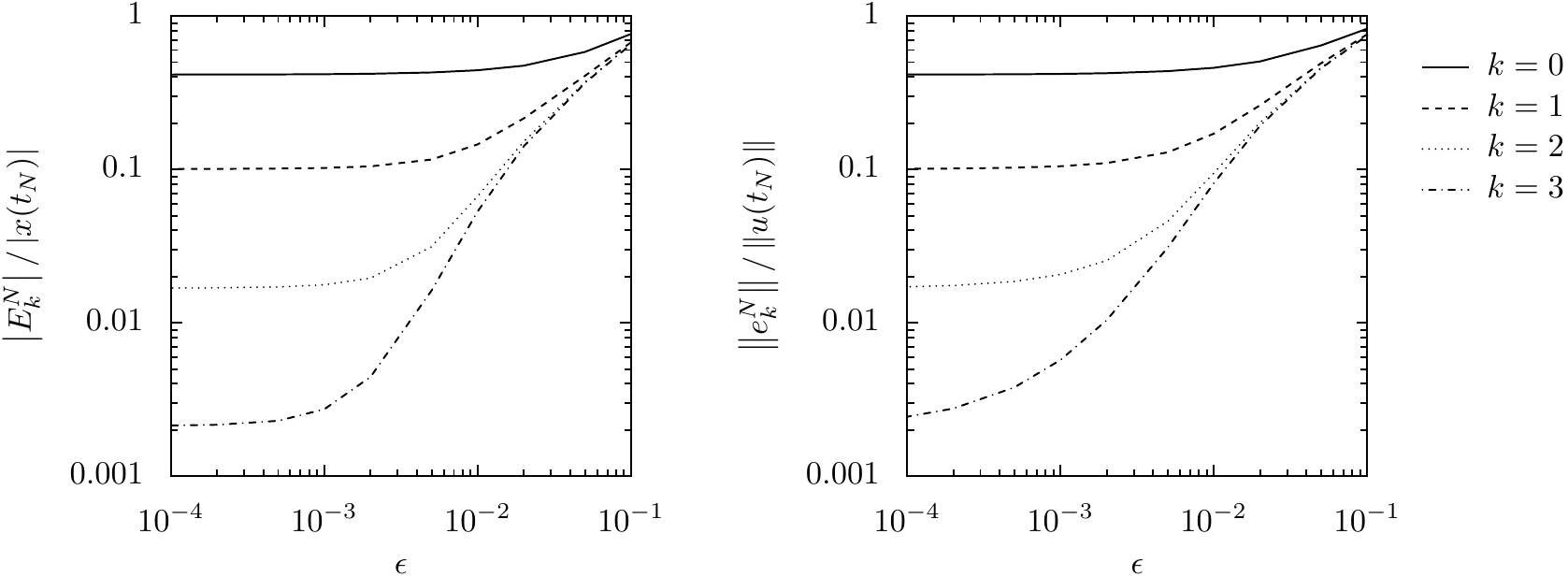}
	\caption{\label{fig:num-fe-lifting}
Algorithm~\ref{algorithm-lifting} for the system~\eqref{eq:toy-problem},
with exact fine-scale propagator and approximate coarse propagator: errors as
a function of $\epsilon$ for different values of $k$ (left: macroscopic error; right: microscopic error).}
\end{figure}

When comparing Figure~\ref{fig:num-fe-lifting} with
Figure~\ref{fig:num-exact-lifting} (in which the macroscopic dynamics
is exactly integrated), we notice that the behavior of the
algorithm is similar for large values of $\epsilon$. For small values of
$\epsilon$, the errors approach an asymptotic value as $\epsilon$ goes to zero
(rather than converging to 0 as in Figure~\ref{fig:num-exact-lifting}),
with an asymptotic value that depends on the number of iterations $k$.
The larger $k$ is, the smaller this asymptotic value is, and the wider
the range of $\epsilon$ for which results of
Figures~\ref{fig:num-fe-lifting} and~\ref{fig:num-exact-lifting} (with
approximate or exact integration at the macroscopic scale) agree. 

This observation is confirmed in Figure~\ref{fig:num-fe-lifting-k},
where we show the errors as a function of the iteration number $k$, $1
\le k \le K=8$, for various values of $\epsilon$.  
We see that the errors first converge exponentially to 0 as $k$
increases, and then reach a plateau. The residual macroscopic
(resp. microscopic) error is of the order of $O(\epsilon^2)$
(resp. $O(\epsilon)$). 

\begin{figure}[htbp]
\includegraphics[scale=0.74]{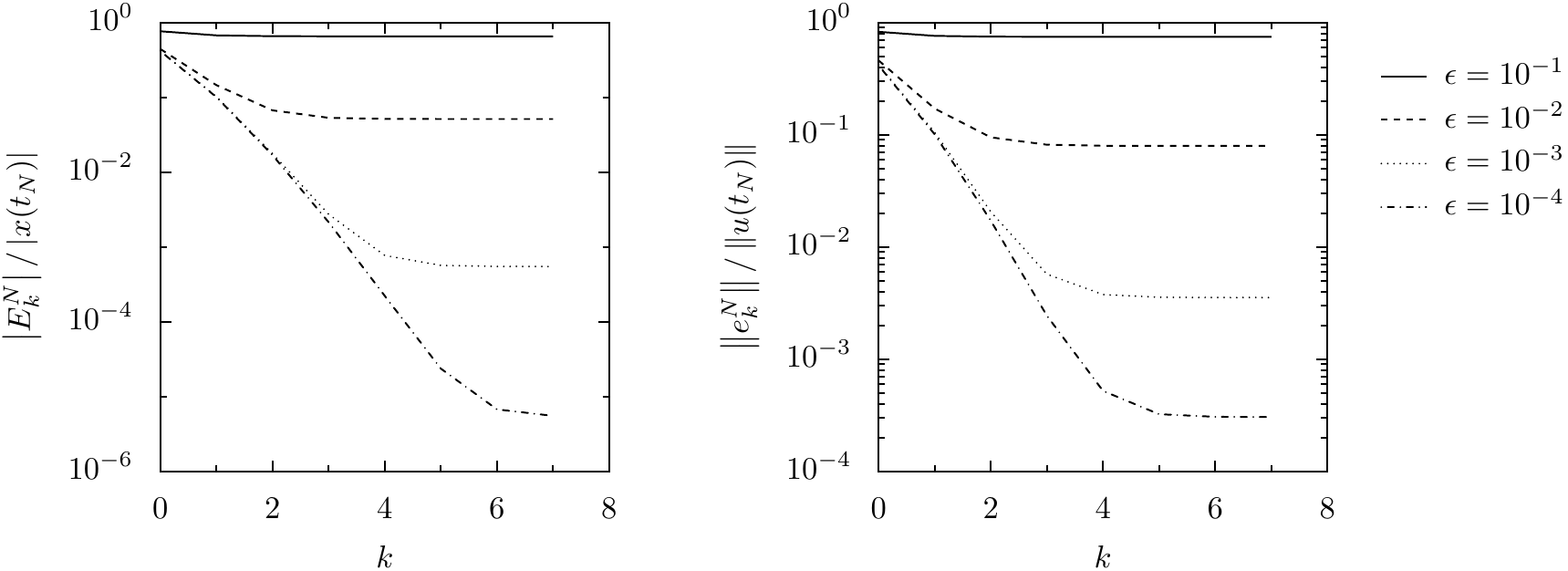}
\caption{\label{fig:num-fe-lifting-k}
Algorithm~\ref{algorithm-lifting} for the system~\eqref{eq:toy-problem},
with exact fine-scale propagator and approximate coarse propagator: errors as
a function of $k$ for different values of $\epsilon$ (left: macroscopic
error; right: microscopic error).} 
\end{figure}

We explain this behavior as follows. The parareal iterations iteratively
correct the approximation made using the coarse propagator. When the
coarse propagator is a forward Euler discretization of the approximate
macroscopic equation, there are two sources of error: a modeling error
(due to the fact that the macroscopic equation~\eqref{eq:lin_macro} is
only an {\em approximation} of the slow part of the reference
model~\eqref{eq:lin_micro}), and a time discretization error. For large
values of $\epsilon$, the modeling error is dominant, and the error
behaves as if the coarse propagator were exact.  For small $\epsilon$,
the time discretization is dominant, and the error becomes therefore
independent of $\epsilon$. Due to the parareal iterations, the time
discretization error is iteratively removed. However, due to the fact
that the reconstruction is performed using the lifting operator $\L$,
the modeling error never vanishes when $\epsilon>0$. 
Hence, when $k$ goes to infinity,
Algorithm~\ref{algorithm-lifting} using an approximate coarse propagator
converges to the solution given by a parareal algorithm with no
time-step discretisation error (this latter has been removed by the
iterations in $k$), but with some modeling error. The solution hence
converges to that given by Algorithm~\ref{algorithm-lifting} with exact
coarse propagation. 

\subsubsection{Algorithm~\ref{algorithm-matching}}

We now consider Algorithm~\ref{algorithm-matching} (for which the
reconstruction is performed using a matching operator $\P$), and set the
maximal number of parareal iterations at
$K=6$. Figure~\ref{fig:num-fe-matching} shows the errors as a function
of $\epsilon$ for the chosen values of $k$. 

\begin{figure}[htbp]
	\includegraphics[scale=0.76]{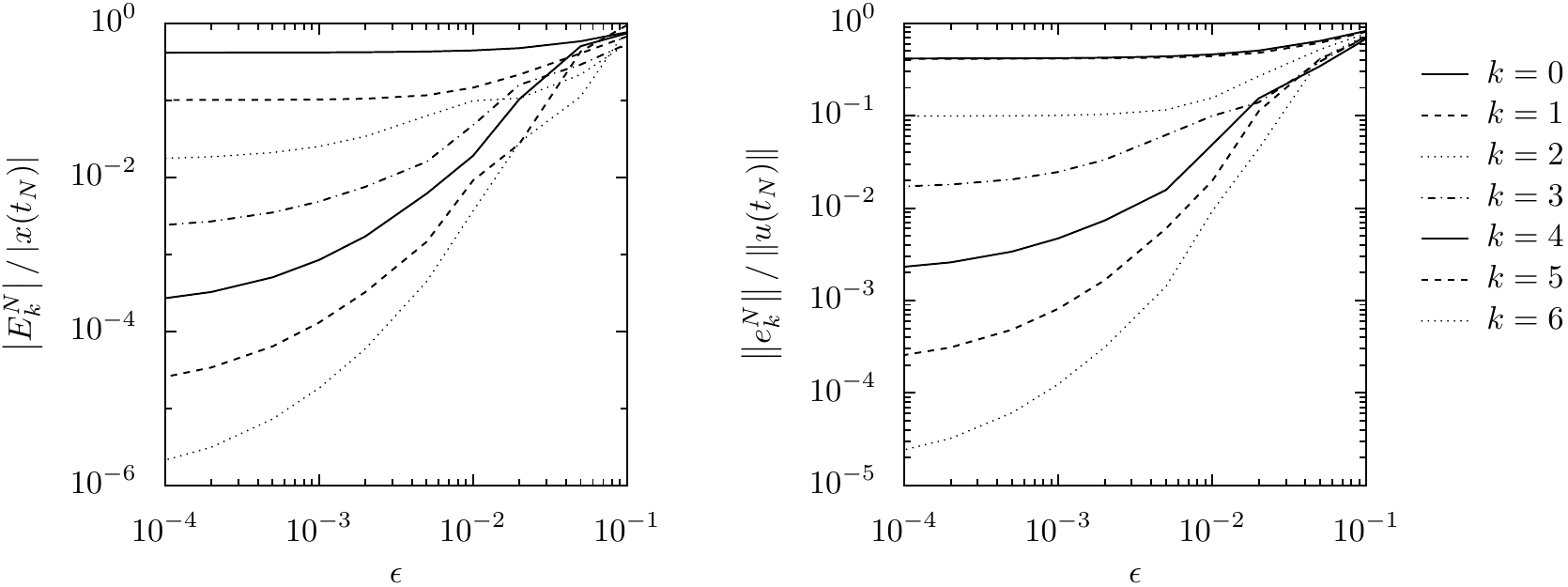}
	\caption{\label{fig:num-fe-matching}
Algorithm~\ref{algorithm-matching} for the system~\eqref{eq:toy-problem},
with exact fine-scale propagator and approximate coarse propagator: errors as
a function of $\epsilon$ for different values of $k$ (left: macroscopic
error; right: microscopic error).}
\end{figure}

We compare Figure~\ref{fig:num-fe-matching} to the corresponding
Figure~\ref{fig:num-exact-matching} (where the macroscopic equation is
exactly integrated). We again notice that, for small
$\epsilon$, the algorithm behaves differently: in particular, the
convergence when $\epsilon$ goes to zero slows down when the macroscopic
equation is only approximately integrated. However, the behavior with
respect to $k$ is left unchanged. We show on 
Figure~\ref{fig:num-fe-matching-k} the evolution of the
errors as a function of the parareal iteration number $k$, $0 \le k \le
K=30$, for a number of fixed values of $\epsilon$. As in 
Figure~\ref{fig:num-exact-matching-k}, the computed
trajectory again converges to the exact microscopic solution up to
machine precision, exponentially with respect to $k$, despite the
presence of time discretization errors at the macroscopic
level. Moreover, comparing these results with those obtained when using an 
exact coarse propagator (see Figure~\ref{fig:num-exact-matching-k}), we
see that only a few extra parareal iterations are needed. 

\begin{figure}[htbp]
	\includegraphics[scale=0.74]{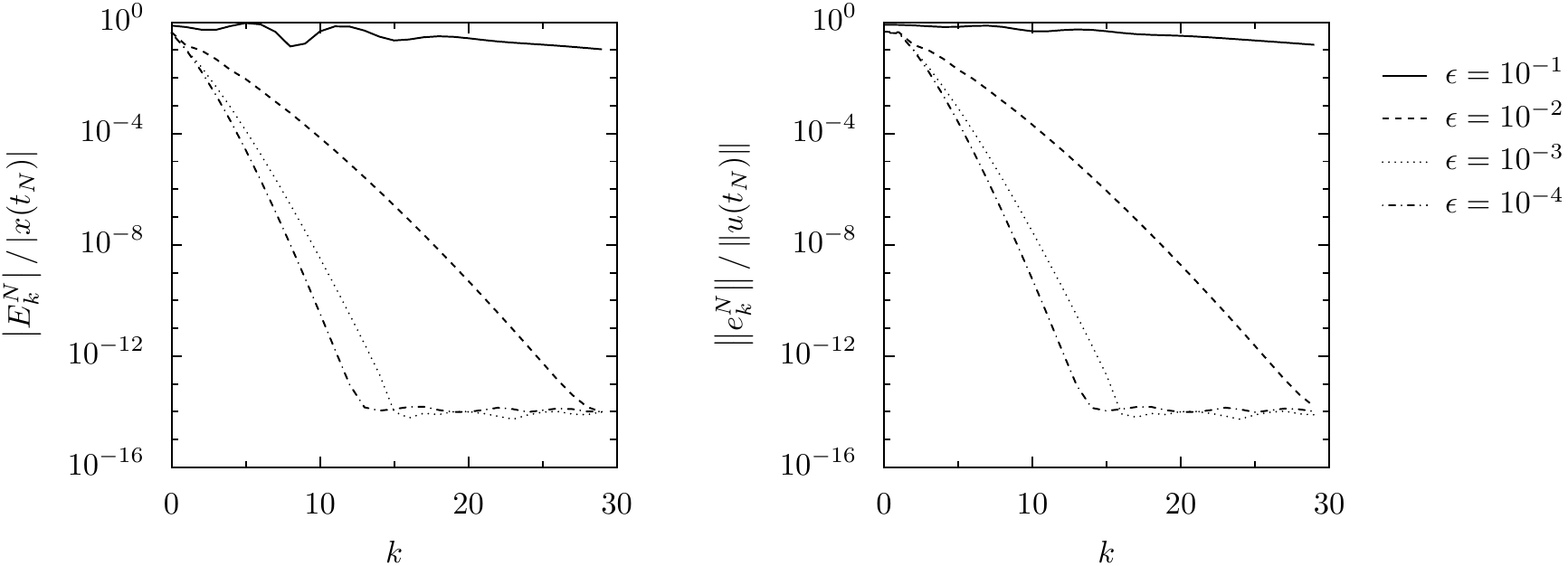}
	\caption{\label{fig:num-fe-matching-k}
Algorithm~\ref{algorithm-matching} for the system~\eqref{eq:toy-problem},
with exact fine-scale propagator and approximate coarse propagator: errors as
a function of $k$ for different values of $\epsilon$ (left: macroscopic
error; right: microscopic error).}
\end{figure}

% \begin{figure}[htbp]
% 	\includegraphics[scale=0.76]{figuren/fig-exact-maday-k}
% 	\caption{\label{fig:num-exact-maday-k}
% Algorithm by Maday for the system~\eqref{eq:toy-problem},
% with exact fine-scale and coarse propagators: errors as
% a function of $k$ for different values of $\epsilon$ (left: macroscopic
% error; right: microscopic error).}
% \end{figure}
% 
% \begin{figure}[htbp]
% 	\includegraphics[scale=0.76]{figuren/fig-fe-maday}
% 	\caption{\label{fig:num-exact-maday}
% Algorithm by Maday for the system~\eqref{eq:toy-problem},
% with exact fine-scale and approximate coarse propagators: errors as
% a function of $\epsilon$ for different values of $k$ (left: macroscopic
% error; right: microscopic error).}
% \end{figure}
% 
% \begin{figure}[htbp]
% 	\includegraphics[scale=0.76]{figuren/fig-fe-maday-k}
% 	\caption{\label{fig:num-exact-maday-k}
% Algorithm by Maday for the system~\eqref{eq:toy-problem},
% with exact fine-scale and approximate coarse propagators: errors as
% a function of $k$ for different values of $\epsilon$ (left: macroscopic
% error; right: microscopic error).}
% \end{figure}

\section{Nonlinear examples\label{sec:nonlin}}

We finally illustrate the performance of our
Algorithm~\ref{algorithm-matching} on two nonlinear examples. Such cases
are not covered by the theoretical analysis of
Section~\ref{sec:analysis}, where we considered a linear problem.  

The first nonlinear example we consider is a straightforward extension
of problem~\eqref{eq:lin_micro}, and reads
\begin{equation}
\label{eq:mic_kwad}
	\dot{x}=-\lambda x - y, \qquad 
	\dot{y}=\dfrac{1}{\epsilon}(x^2-y),
\end{equation} 
which is of the form~\eqref{eq:micro}. The corresponding macroscopic
model is 
\begin{equation}
\label{eq:mac_kwad}
	\dot{X}=-\lambda X - X^2. 
\end{equation} 
We use Algorithm~\ref{algorithm-matching} to integrate this system. The
fine propagator $\F_{\Delta t}$ is a forward Euler scheme
for~\eqref{eq:mic_kwad} with the time step $\delta t = 10^{-5}$. The
coarse propagator $\C_{\Delta t}$ is a forward Euler scheme
for~\eqref{eq:mac_kwad} with the time step $\Delta t = 10^{-1}$ (which
is equal to the parareal time step). The lifting operator reads $\L(X) =
(X,X^2)$, and the matching operator is again given
by~\eqref{eq:projection-good}. 
The remaining parameters are chosen
identical to those in the previous experiments. On
Figure~\ref{fig:kwad}, we show the error as a function of $\epsilon$ for
different values of $k$. Comparing that figure with
Figure~\ref{fig:num-exact-matching}, we see that
Algorithm~\ref{algorithm-matching} performs equally well on this
nonlinear case as on the linear problem considered in
Section~\ref{sec:num}.

\begin{figure}[htbp]
	\includegraphics[scale=0.74]{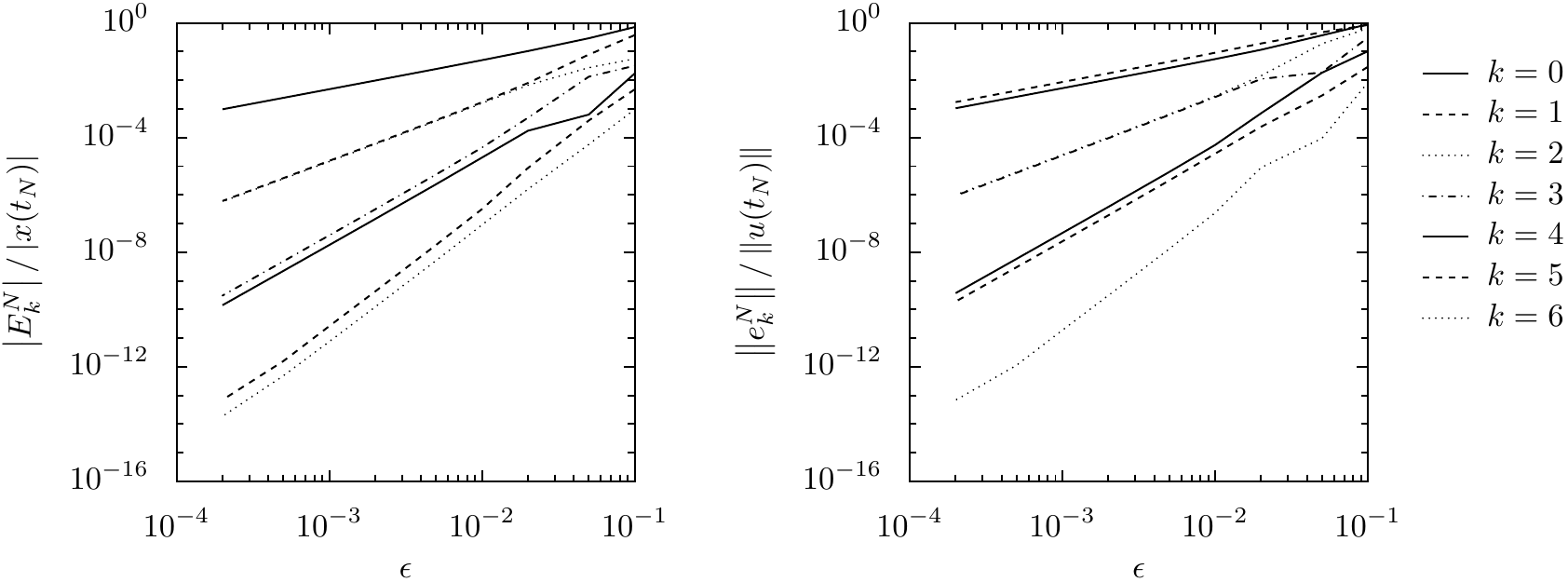}
	\caption{\label{fig:kwad}
Algorithm~\ref{algorithm-matching} for the system~\eqref{eq:mic_kwad},
with parareal time step $\Delta t = 10^{-1}$: errors as
a function of $\epsilon$ for different values of $k$ (left: macroscopic
error; right: microscopic error).}
\end{figure}

The second nonlinear example we consider is the so-called Brusselator
problem, which was already considered in
e.g.~\cite{gear2003projective}. It reads 
\begin{align}
	\dot{x}_1&=A-(y+1)x_1+x_1^2x_2,\nonumber\\
	\dot{x}_2&=yx_1-x_1^2x_2,\label{eq:brus_mic}\\
	\dot{y}&=\dfrac{1}{\epsilon}(B_0-y)-yx_1.
\nonumber
\end{align}
It models the evolution of the concentration of three chemical species.
The concentration of $y$ is reduced via reaction with
$x$, but restored to its base level $B_0$ with a characteristic time of
the order of $\epsilon$. We choose $A=1$ and $B_0=3$.  The fine propagator $\F_{\Delta t}$ is a forward Euler
discretization of~\eqref{eq:brus_mic} with the time step $\delta t =
10^{-5}$. The coarse propagator $\C_{\Delta t}$ is a forward Euler
discretization of the macroscopic model
$$
\dot{x}_1=A-(B_0+1)x_1+x_1^2x_2, \qquad \dot{x}_2=B_0 \, x_1-x_1^2x_2,
$$
with the time step $\Delta t=10^{-1}$ (equal to the parareal time step).  
This system thus has two slow variables and one fast one: $u=(x,y)$, with $x=(x_1,x_2)$. Note that this case does not enter our theoretical framework for two
reasons: (i) the example is nonlinear; and (ii) the equation for $y$ in
the microscopic model is not purely a fast equation (the second term in
the right-hand side of the equation for $\dot{y}$ in~\eqref{eq:brus_mic}
is {\em not} scaled by $\epsilon^{-1}$).

We show on Figure~\ref{fig:brus} the results obtained. 
Algorithm~\ref{algorithm-matching} again performs very well. Actually,
on this problem, the convergence behavior of
Algorithm~\ref{algorithm-matching} closely resembles that of
Algorithm~3: at parareal iteration $k$, the order of convergence (in
terms of $\epsilon$) seems to be equal to $k$, both for the macroscopic
and the microscopic errors. 

\begin{figure}[htbp]
	\includegraphics[scale=0.74]{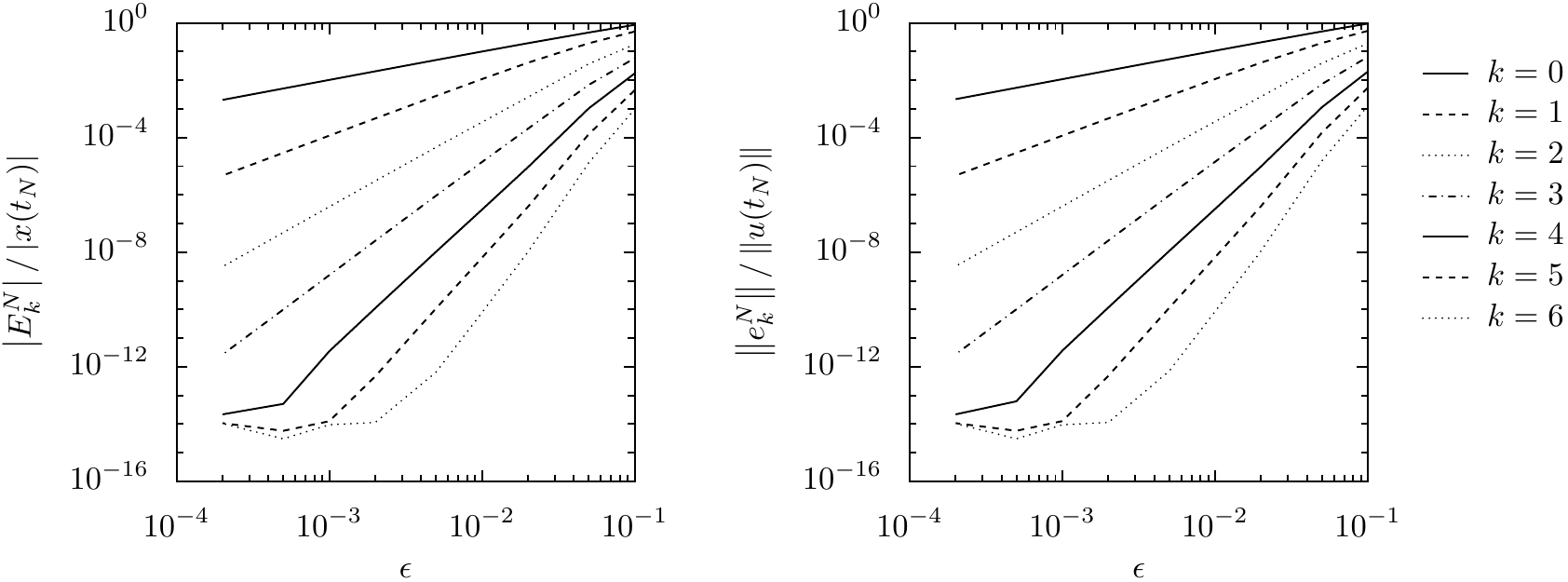}
	\caption{\label{fig:brus}
Algorithm~\ref{algorithm-matching} for the system~\eqref{eq:brus_mic},
with parareal time step $\Delta t = 10^{-1}$: errors as
a function of $\epsilon$ for different values of $k$ (left: macroscopic
error; right: microscopic error).}
\end{figure}

\section{Discussion and conclusions\label{sec:concl}}

We have introduced and analyzed two micro-macro parareal algorithms for the
time-parallel integration of singularly perturbed ordinary differential
equations, and provided a numerical analysis for the case where the
problem is linear and the
coarse and fine-scale propagators integrate the macroscopic,
resp.~microscopic models exactly. The analysis shows that, when an
\emph{appropriate matching operator} is used to update the microscopic
state after correction of the macroscopic state (which corresponds to
using Algorithm~\ref{algorithm-matching}), the rate of convergence (in
terms of the modelling error $\epsilon$, which 
quantifies the time scale separation between the microscopic and the
macroscopic evolutions) improves at each parareal
iteration, and is roughly equal to $\epsilon^{k/2}$. We have illustrated
this theoretical result with 
numerical experiments, and also numerically investigated the effect of
using a numerical scheme to integrate the macroscopic model (thereby
introducing some discretization error). 
The results show that the proposed micro-macro parareal algorithm,
Algorithm~\ref{algorithm-matching},  
is robust with respect to time discretization errors at the macroscopic level. 
It can hence be viewed as 
an interesting way of using an approximate macroscopic model to
speed up simulations of high-dimensional multiscale systems of
singularly perturbed ordinary differential equations, provided that
special care is taken when transferring information between the
microscopic and macroscopic levels of description. 

Several questions remain open. First, while the analysis reveals that it
is important to choose the parareal time step $\Delta t$ sufficiently
large (to average out the initial time boundary layers in the full
microscopic dynamics), the dependence on $\Delta t$ of the
numerical error and the convergence rate have not been analyzed.  
In particular, one may expect an optimal time step $\Delta t$ to exist
that leads to a required accuracy with a minimal cost. Assume again
(as for the original parareal algorithm, see the introduction) that the cost
of a single evaluation of the fine-scale propagator $\F_{\Delta t}$ is much
larger than the cost of propagating the macroscopic system using
$\C_{\Delta t}$ over the complete time range $[0,T]$ (This assumption is
all the more justified as the macroscopic system is a low-dimensional
problem compared to the microscopic problem). Then the cost of the
parareal algorithm, after $K$ iterations, is proportional to $K
\Delta t / \epsilon$ (We have assumed that the cost of $\F_{\Delta t}$ is
proportional to $\Delta t / \epsilon $, since we need to use a time step
of the order of $\epsilon$ over a time range of length $\Delta
t$). This cost is to be compared with the cost of the full microscopic
sequential integration, which is proportional to $N\Delta
t/\epsilon$. The computational speed-up is thus $N/K$. We saw on 
Figure~\ref{fig:num-exact-matching} that, for reasonably small values
of $\epsilon$, results obtained at the iteration $K = 6$ were
satisfactory. For the test-case considered in
Section~\ref{sec:num_algo2}, the computational speed-up is thus 
$$
\frac{N}{K} = \frac{T/\Delta t}{K} = 16.6.
$$

Second, we expect Algorithm~\ref{algorithm-matching} to extend to more general
dissipative systems. As pointed out above, we considered here the simple
linear 
problem~\eqref{eq:lin_micro} to focus on the issues stemming from using
two different levels of description of the same system. We have already
checked in Section~\ref{sec:nonlin} that
Algorithm~\ref{algorithm-matching} behaves equally well on nonlinear
systems of singularly perturbed ODEs. We currently
investigate the extension
of the algorithm to a setting, motivated by molecular simulations, where
the reference model is a high-dimensional stochastic differential
equation (modeling the evolution of all the degrees of freedom of the
atomistic system) and the macroscopic model is the effective dynamics of
the slow 
component of the microscopic model, derived under time scale separation
assumptions following~\cite{legoll2010effective,legoll2011lncse}.

\section*{Acknowledgements}

FL and TL thank Sorin Mitran for enlightening discussions that eventually
led to this work. All authors thank the anonymous referees for their
comments that lead to a substancial improvement of the manuscript. 
Part of this work was performed during a research stay of GS
at CERMICS (ENPC -- Paris), when he was a Postdoctoral Fellow of the Research Foundation
-- Flanders (FWO). GS warmly thanks the whole CERMICS team for
its hospitality, and both CERMICS and FWO for funding this stay. 
This work was (partially) completed while FL and TL were
visiting the Institute for Mathematical Sciences, National University of
Singapore in 2012. 
This work was partially supported by the Research Council of the K.U. Leuven
through grant OT/09/27, by the Interuniversity Attraction Poles
Programme of the Belgian Science Policy Office under grant IUAP/V/22,
and by the Agence Nationale de la
Recherche under grant ANR-09-BLAN-0216-01 (MEGAS).
The scientific responsibility rests with its authors.

%\fred{update ref~\cite{dai2010symmetric} des que possible}

%\fred{la ref~\cite{bal2003parallelization} n'est toujours que soumise
%  sur le site de Bal. Je ne sais pas d'où vient le 2003}

%\fred{la ref~\cite{gander2008nonlinear} est bien 2008}

\bibliographystyle{plain}
\bibliography{bibliography}

\appendix

\section{Proofs of Lemma~\ref{lem:slaving} and
  Corollary~\ref{cor:slaving}} 
\label{sec:appendix}

Before proving Lemma~\ref{lem:slaving} and Corollary~\ref{cor:slaving},
we start with a preliminary result. Here and in all what follows,
$\|\cdot\|$ denotes the Euclidean norm when applied to vectors, and the
associated operator norm when applied to matrices.

\smallskip

\begin{lem}
\label{lem:jordan}
Let $M$ be a matrix in $\RR^{d \times d}$ such that the real part of the
spectrum of $M$ is positive. Then, there exist $C>0$ and $\mu>0$ such
that, for all time $t \ge 0$,
\begin{equation}
\label{eq:bound_exp_A}
\| \exp(-Mt) \| \le C \exp(-\mu t).
\end{equation}
One can choose $\mu=\inf_{\nu \in \sigma(M)} \Re(\nu) /2$, where
$\sigma(M)$ denotes the spectrum of $M$. We also have
\begin{equation}
\label{eq:bound_A_inv}
\left\| M^{-1} \right\| \le \frac{C}{\mu}.
\end{equation}
\end{lem}

\begin{proof} 
We introduce the Jordan form of the matrix $M$. Let us assume for 
simplicity of notation that $M$ has only two Jordan blocks associated to
two complex eigenvalues $\lambda_1 \neq \lambda_2$ with 
$0 < \Re(\lambda_1) \leq \Re(\lambda_2)$. The generalization to any
number of Jordan blocks is straightforward. Let us denote $\mu=\inf_{\nu
  \in \sigma(M)} \Re(\nu) /2 = \Re(\lambda_1)/2 >0$.

Since $M$ has only two Jordan blocks, there exists an
invertible matrix $Q \in \RR^{d \times d}$ such that
$$
M = Q^{-1} \left[ 
\begin{array}{cc} N_{d_1}(\lambda_1) & 0 \\ 0 & N_{d_2}(\lambda_2) 
\end{array} \right] Q,
$$
where $d_1+d_2 = d$, and, for any $m \in \N^\star$ and any $\lambda
\in \CC$,
$$
N_m(\lambda) = \left[
\begin{array}{cccccc}
\lambda & 1 & 0 & 0 & & \\
0 & \lambda & 1 & 0 & & \\
& & \dots & & & \\
& & & \dots & & \\
&&& 0 & \lambda & 1 \\
&&&& 0 & \lambda \\
\end{array}
\right] \in \RR^{m \times m}.
$$
We compute, for any $t \in \RR$,
\begin{equation}
\label{eq:decompo_A}
\exp(Mt) 
=
Q^{-1}
\left[ \begin{array}{cc} 
\exp(\lambda_1 t) 
P_{d_1}(t) & 0 
\\
0 & \exp(\lambda_2 t) P_{d_2}(t)
\end{array} \right]
Q,
\end{equation}
where, for any $m \in \N^\star$ and any $t \in \RR$,
$$
P_m(t) = \left[
\begin{array}{ccccccc}
1 & t & t^2/2 & t^3/6 & \dots & \dots & t^{m-1}/((m-1)!)\\
0 & 1 & t & t^2/2 & \dots & \dots & t^{m-2}/((m-2)!)\\
& & \dots & & & & \\
& & & \dots & & & \\
&&& & 0 & 1 & t \\
&&&& & 0 & 1 \\
\end{array}
\right] \in \RR^{m \times m}.
$$
Since $P_m(t)$ is a matrix with entries which are polynomial functions
of $t$, there exists a constant $C$ that only depends on the matrix $M$
such that
$$
\forall t \ge 0, \quad
\| P_{d_1}(-t) \| + \| P_{d_2}(-t) \| \leq 
C \exp( \mu t).
$$
We then infer from~\eqref{eq:decompo_A} that there exists a constant
that only depends on $M$ such that 
\begin{equation*}
\forall t \ge 0, \quad
\left\| \exp(-Mt) \right\|
\leq
C \exp( -\mu t).
\end{equation*}
This yields~\eqref{eq:bound_exp_A}. Then,~\eqref{eq:bound_A_inv} is
obtained using the fact that 
$$
\left\| M^{-1} \right\|
= 
\left\| \int_0^\infty \exp(-Mt) \,dt \right\|
\le  
\int_0^\infty \left\| \exp(-Mt)\right\| \,dt
\le 
C  \int_0^\infty \exp(-\mu t) \,dt
=
\frac{C}{\mu}.
$$
This concludes the proof of Lemma~\ref{lem:jordan}.
\end{proof}

\medskip

We are now in position to prove Lemma~\ref{lem:slaving} and
Corollary~\ref{cor:slaving}. 

\subsection*{Proof of Lemma~\ref{lem:slaving}}

We start by writing 
\begin{align}
\dot{z}
&=\dot{y}-A^{-1}q\,\dot{x}
\nonumber
\\
&
= \frac{1}{\epsilon}\left(qx-Ay\right)-A^{-1}q\left(\alpha x+p^Ty\right)
\nonumber
\\
&=-\left[\frac{A}{\epsilon}+ \left( A^{-1}q \right) p^T\right] z -
\lambda \left( A^{-1} q \right) x,
\label{eq:debut}
\end{align}
where $\lambda$ is defined by~\eqref{eq:lin_macro}. Introducing
$$
M^\epsilon := A + \epsilon
\left( A^{-1}q \right) p^T \in \RR^{(d-1) \times (d-1)},
\quad
V := \lambda \left( A^{-1} q \right) \in \RR^{(d-1)},
$$
we recast~\eqref{eq:debut} as
\begin{equation}
\label{eq:dyn_z}
\dot{z}
=
- \frac{ M^\epsilon }{\epsilon} z - V x.
\end{equation}
% cf. les valeurs propres sont continues en fonction des coefficients
From the definition of $M^\epsilon$, and in view of
Assumption~\eqref{ass:expo-stable}, it is clear that there exists a
critical value $\epsilon_0(A,q,p)$ such that for all
$\epsilon<\epsilon_0(A,q,p)$, the matrix $M^\epsilon$ has a spectrum
with a real part bounded from below by $\lambda_-/2 > 0$, where
$\lambda_-$ is independent of $\epsilon$. Up to a modification of
$\epsilon_0(A,q,p,\alpha)$, the same property holds true for the matrix
$M^\epsilon+ \epsilon \lambda\Id$ that will appear below (where
$\lambda$ is defined by~\eqref{eq:lin_macro}). In the sequel of the
proof, we will systematically work with
$\epsilon<\epsilon_0(A,q,p,\alpha)$. 

By explicit integration of~\eqref{eq:dyn_z}, we have
\begin{equation}
z(t) - \exp(-M^\epsilon t/ \epsilon) z_0 = - 
\int_0^t \exp\left[-M^\epsilon(t-s) / \epsilon \right] \, V \, x(s)
\, ds.
\label{eq:z-t0}
\end{equation}
From~\eqref{eq:lin_micro}, we have $\dot{x} = \lambda x + p^T z$.
Using equation~\eqref{eq:z-t0}, we thus obtain
\begin{align}
x(t) - x_0 \exp(\lambda t) 
&= p^T\int_0^t \exp\left(\lambda(t-s)\right) z(s) ds 
\nonumber\\
&= p^T \int_0^t \exp\left(\lambda(t-s)\right)
\exp\left(-M^\epsilon s / \epsilon \right) z_0 ds 
\nonumber\\
&\quad - p^T \int_0^t \exp\left(\lambda(t-s)\right) \int_0^s
\exp\left(-M^\epsilon(s-r) / \epsilon \right) V x(r) drds. 
\label{eq:x-t}
\end{align}
To bound the first term of~\eqref{eq:x-t}, we write, using 
Lemma~\ref{lem:jordan},
\begin{align}
&\left\| \int_0^t \exp\left(\lambda(t-s)\right) \exp\left(-M^\epsilon s / \epsilon \right) ds \right\|
\nonumber
\\
&
= \exp(\lambda t) \left\| \int_0^t \exp\left(-\left(M^\epsilon/ \epsilon
      +\lambda\Id\right) s\right) ds \right\|
\nonumber
\\
& \le \left\| \left(M^\epsilon / \epsilon +\lambda\Id\right)^{-1} \right\|
\left\| \exp(\lambda t) \Id - \exp\left(-\left(M^\epsilon / \epsilon \right) t\right)
\right\|
\nonumber
\\
& \le \eps \left\| \left(M^\epsilon + \eps \lambda\Id\right)^{-1} \right\|
\Big( \left\| \exp(\lambda t) \Id \right\| + 
\left\| \exp\left(-\left(M^\epsilon / \epsilon \right) t\right) \right\|
\Big) 
\nonumber
\\
& \le \eps \frac{C(A,q,p,\alpha)}{\lambda_-/4} \Big( \exp(\lambda T) + 
C(A,q,p,\alpha) \exp\left(-\lambda_- t / (4\epsilon) \right) \Big)
\nonumber 
\\
& \le C(A,q,p,\alpha,T) \, \epsilon,
\label{eq:second-term}
\end{align}
when $\epsilon \leq \epsilon_0(A,q,p,\alpha)$. Turning to the second term
of~\eqref{eq:x-t}, we use Fubini's theorem, and write 
\begin{eqnarray*}
& & \int_0^t \exp\left(\lambda (t-s)\right) \int_0^s
\exp\left(-M^\epsilon(s-r)/ \epsilon \right) V \, x(r) \, drds
\\
&=&
\exp(\lambda t) \int_0^t \exp(M^\epsilon r/ \epsilon) \left[ \int_r^t
  \exp\left(-(M^\epsilon / \epsilon +\lambda\Id)s\right) ds \right] V \, x(r)
\, dr
\\
&=&
\exp(\lambda t) \int_0^t \exp(M^\epsilon r / \epsilon)
\Bigl[ \exp\left(-(M^\epsilon / \epsilon +\lambda\Id)r\right) -
\exp\left(-(M^\epsilon / \epsilon +\lambda\Id)t\right) \Bigr]
\\
&& \hspace{5cm}
\times \left(M^\epsilon / \epsilon +\lambda\Id\right)^{-1} V \, x(r)
\, dr
\\
&=&	
\int_0^t \bigl[
\exp\left(\lambda (t-r) \right) \Id 
-\exp\left(-M^\epsilon (t-r) / \epsilon \right)
\bigr] \left(M^\epsilon / \epsilon +\lambda\Id\right)^{-1} V \, x(r) \, dr.
\end{eqnarray*}
Therefore, for $\epsilon \leq \epsilon_0(A,q,p,\alpha)$, using
Lemma~\ref{lem:jordan}, we obtain
\begin{align}
& \left\| \int_0^t \exp\left(\lambda (t-s)\right) \int_0^s
  \exp\left(-M^\epsilon(s-r) / \epsilon \right) V x(r)drds
\right\|
\nonumber
\\
&\le \left\| \left( M^\epsilon / \epsilon +\lambda\Id \right)^{-1} \right\| \
\left\| V \right\| \ \sup_{0\le r\le t} |x(r)|
\int_0^t \! \! \! \left\| \exp\left(\lambda (t-r) \right) \Id 
	-\exp\left(-M^\epsilon (t-r) / \epsilon\right) \right\| dr
\nonumber\\
&\le 
\eps \left\| \left( M^\epsilon +\eps \lambda\Id \right)^{-1} \right\|
C(A,q,p,\alpha,T) \,  
\sup_{0\le r\le t}|x(r)| 
\nonumber\\
&\le C(A,q,p,\alpha,T) \, \epsilon \, \sup_{0\le r\le t}|x(r)|.
\label{eq:third-term}
\end{align}
Combining equations~\eqref{eq:x-t},~\eqref{eq:second-term}
and~\eqref{eq:third-term}, we get 
\begin{align*}
|x(t)-x_0\exp(\lambda t)| 
& \le 
C(A,q,p,\alpha,T) \| p \| \, \| z_0 \| \, \eps 
+
C(A,q,p,\alpha,T) \, \| p \| \, \eps \sup_{0\le r\le t}|x(r)|
\\
& \le 
C(A,q,p,\alpha,T) \epsilon \left( \| z_0 \| 
+
\sup_{0\le r\le t}|x(r)| \right),
\end{align*}
and hence,
\begin{align*}
& \sup_{0\le t \le T} |x(t)-x_0\exp(\lambda t)| 
\\
&\le 
C(A,q,p,\alpha,T) \ \epsilon \left( 
\| z_0 \| + \sup_{0\le r\le T}|x(r)| \right)
\\
& \le 
C(A,q,p,\alpha,T) \ \epsilon \left(
\| z_0 \| + |x_0| + \sup_{0\le r\le T} |x(r) - x_0\exp(\lambda r)| \right).
\end{align*}
We deduce that, for $\epsilon \leq \epsilon_0(A,q,p,\alpha,T)$,
\begin{align}
\sup_{0\le t \le T} |x(t)-x_0\exp(\lambda t)| 
&\le
\frac{C(A,q,p,\alpha,T) \, \epsilon}{1-C(A,q,p,\alpha,T) \, \epsilon} 
\ (\|z_0\|+|x_0|)
\nonumber
\\
&\leq
\overline{C}(A,q,p,\alpha,T) \ \epsilon \ (\|z_0\|+|x_0|).
\label{eq:est-x}
\end{align}
This proves~\eqref{eq:x-est-lem}. 

\medskip

We now turn to proving the bound~\eqref{eq:z-est-lem-bl} on $z(t)$. 
Introducing
$$
B := \left( A^{-1}q \right) p^T \in \RR^{(d-1) \times (d-1)},
$$
we now recast~\eqref{eq:debut} as
$$
\dot{z}
=
-\left[\frac{A}{\epsilon}+ B \right] z - V x
=
- \frac{A}{\epsilon} z - \left[ B z + V x \right].
$$
By explicit integration, we have
\begin{equation}
z(t) - \exp(-At/\epsilon) z_0 = -
\int_0^t \exp\left[-A(t-s)/\epsilon\right] \left( B z(s) + V x(s)
\right) ds.
\label{eq:z-t}
\end{equation}
Using~\eqref{eq:bound_exp_A} and Assumption~\eqref{ass:expo-stable}, we obtain
\begin{eqnarray*}
\left\| z(t) - \exp(-At/\epsilon) z_0 \right\|
&\leq & C(A)
\int_0^t \exp \left[ \lambda_- \dfrac{s-t}{2\epsilon} \right] 
\left\| B z(s) + V x(s) \right\| ds
\nonumber
\\
&\leq &
C(A,q,p,\alpha) 
\int_0^t 
\exp\left[ \lambda_- \dfrac{s-t}{2\epsilon} \right]
\left( \| z(s) \| + | x(s) | \right)
ds
\nonumber
\\ 
&\leq &
C(A,q,p,\alpha) \sup_{s \in [0,t]} \left( \| z(s) \| + | x(s) | \right)
\dfrac{2 \epsilon}{\lambda_-}
\nonumber
\\ 
&\leq &
\epsilon C(A,q,p,\alpha) \left(  
\sup_{s \in [0,t]} \| z(s) - \exp(-As/\epsilon) z_0 \| 
\right.
\nonumber
\\
&& \hspace{1cm} \left. + 
\sup_{s \in [0,t]} \| \exp(-As/\epsilon) z_0 \| 
+
\sup_{s \in [0,t]} | x(s) | 
\right).
\nonumber
\end{eqnarray*}
Taking the supremum over $t \in [0,T]$, we obtain, for $\epsilon \leq
\epsilon_0(A,q,p,\alpha)$, 
\begin{eqnarray*}
\sup_{t \in [0,T]} \| z(t) - \exp(-At/\epsilon) z_0 \| 
&\leq &
\frac{\epsilon C(A,q,p,\alpha)}{1 - \epsilon C(A,q,p,\alpha)}  
\left(  
\sup_{s \in [0,T]} \| \exp(-As/\epsilon) z_0 \| 
+
\sup_{s \in [0,T]} | x(s) | 
\right)
\nonumber
\\
&\leq &
\frac{\epsilon C(A,q,p,\alpha)}{1 - \epsilon C(A,q,p,\alpha)}  
\left(  
C(A) \| z_0 \| 
+
\sup_{s \in [0,T]} | x(s) | 
\right).
\nonumber
\end{eqnarray*}
We then deduce from~\eqref{eq:est-x} that, for $\epsilon \leq
\epsilon_0(A,q,p,\alpha,T)$,
 \begin{eqnarray}
\sup_{t \in [0,T]} \| z(t) - \exp(-At/\epsilon) z_0 \| 
&\leq &
\frac{\epsilon C(A,q,p,\alpha,T)}{1 - \epsilon C(A,q,p,\alpha)}  
\left(  
\| z_0 \| + | x_0 | 
\right)
\nonumber
\\
&\leq &
\epsilon \overline{C}(A,q,p,\alpha,T)
\left( \| z_0 \| + | x_0 | \right).
\label{eq:z-t4}
\end{eqnarray}
This proves~\eqref{eq:z-est-lem-bl}. 

\medskip

We finally turn to
proving~\eqref{eq:z-est-lem}. Using~\eqref{eq:bound_exp_A}, we see that 
$$
\| \exp(-At/\epsilon) \| \leq 
C(A) \exp( -\lambda_- t/(2\epsilon)),
$$
thus, for times $\dps t \geq t^{\rm BL}_\eps = \frac{2\eps}{\lambda_-}
\ln(1/\eps)$, we have 
$\| \exp(-At/\epsilon) \| \leq C(A) \epsilon$. We then deduce
from~\eqref{eq:z-t4} the bound~\eqref{eq:z-est-lem}. This concludes the
proof of Lemma~\ref{lem:slaving}. 

\medskip

\subsection*{Proof of Corollary~\ref{cor:slaving}}

The first assertion follows directly from~\eqref{eq:x-est-lem} and the
fact that $\| z_0 \| \leq \| y_0 \| + C | x_0 |$. The second assertion
follows from~\eqref{eq:z-est-lem} and~\eqref{eq:x-est-cor}.

\end{document}